\documentclass{amsart}
\usepackage{amsmath,amsfonts,amsthm,amssymb,nameref,mathtools,mathbbol,mathrsfs,amscd}
\usepackage[margin=1.3in]{geometry}
\usepackage{tikz} 
\usetikzlibrary{calc,automata} 
\usepackage{color} 
\usepackage{xfrac}
\usepackage{enumerate}
\usepackage{pinlabel}
\usepackage{graphicx}
\usepackage{caption}
\usepackage{subcaption}
\usepackage{float}
\newcommand{\ax}{\mathrm{Ax}}

\usepackage[pdftex]{hyperref}
\usepackage[all]{hypcap}
\usepackage{microtype} 

\usepackage{overpic}

\usepackage{ifthen}

\newcommand{\showcomments}{yes}

\newsavebox{\commentbox}
{\ifthenelse{\equal{\showcomments}{yes}}%
{\footnotemark
    \begin{lrbox}{\commentbox}
    \begin{minipage}[t]{1.25in}\raggedright\sffamily\upshape\tiny
    \footnotemark[\arabic{footnote}]}%
{\begin{lrbox}{\commentbox}}}%
{\ifthenelse{\equal{\showcomments}{yes}}%
{\end{minipage}\end{lrbox}\marginpar{\usebox{\commentbox}}}%
{\end{lrbox}}}

\newtheorem{Theorem}{Theorem}[section]
\newtheorem{Lemma}[Theorem]{Lemma}

\newtheorem{Proposition}[Theorem]{Proposition}
\newtheorem{Corollary}[Theorem]{Corollary}

\theoremstyle{definition}
\newtheorem{Definition}[Theorem]{Definition}
\newtheorem{Remark}[Theorem]{Remark}

\newtheorem{Non-examples}[Theorem]{Non-examples}

\newtheorem{Question}[Theorem]{Question}

\usepackage{soul}
\begin{document}

\title[Counting double cosets  with applications]{\footnotesize Counting double cosets  with application to  generic 3-manifolds}

\author{Suzhen Han}
\address{Academy of Mathematics and Systems Science\\
Chinese Academy of Sciences\\ , Beijing 100190, P. R. China. }
\email{suzhenhan@pku.edu.cn}

\author{Wenyuan Yang}
\address{Beijing International Center for Mathematical Research\\
Peking University\\
 Beijing 100871, China P.R.}
\email{wyang@math.pku.edu.cn}

\author{Yanqing Zou}
\address{ School of Mathematical Sciences and  Shanghai Key Laboratory of PMMP \\East China Normal University \\
Shanghai 200241, China P.R.}
\email{yqzou@math.ecnu.edu.cn}

\date{\today}

\maketitle

\vspace{-2em}
\begin{abstract}
We study the growth of double cosets in the class of groups with contracting elements, including relatively hyperbolic groups, CAT$(0)$ groups and mapping class groups among others.
Generalizing a recent work of Gitik and Rips about hyperbolic groups, we prove that the double coset growth of two Morse subgroups of infinite index is comparable with the orbital growth function. The same result is further obtained for a more general class of subgroups whose limit sets are proper subsets in the entire limit set of the ambient group.  

As an application, we confirm a conjecture of  Maher that hyperbolic 3-manifolds are exponentially generic in the set of 3-manifolds built from Heegaard splitting using complexity in Teichm\"{u}ller metric.
\end{abstract}


\section{Introduction}

Since the seminar work of Milnor, the growth of groups has been a subject of research  for a long time with a still-growing huge body of results in the literature. While the growth of other objects in groups such as subgroups and their cosets has also been investigated by various authors, the double coset growth for a pair of subgroups was proposed by de la Harpe in his book \cite{Harpe00} (see also \cite[Rem.~2]{Pau13}), but is recently receiving attention from the GGT community. Gitik-Rips \cite{GitikRips20} showed in the class of hyperbolic groups  that, on the one hand, the growth function of double cosets for infinite index quasi-convex subgroups is comparable with the growth function of the ambient  group, and on the other hand, any reasonable double coset growth function could be realized for certain non-quasi-convex (normal) subgroups. In dynamical system, the double coset growth has actually been studied under the guise of counting shortest essential arcs between   convex sub-manifolds in the ambient Riemannian manifold. We refer the reader to   the work of Parkkonen-Paulin \cite{PP17}, and the survey \cite{PP16} for more relevant results and applications. Particularly, the first part of Gitik-Rips's results was a special case of \cite[Theorem 3.1]{HP04} which holds with greater generality for discrete group actions on Gromov hyperbolic spaces.  From  a view of coarse geometry, the main purpose of this paper is to provide a far-reaching generalization  of these results in the class of  groups with  contracting elements, which includes many negatively curved groups such as (relatively) hyperbolic groups, mapping class groups and CAT(0) groups with rank-1 elements. This shall give applications to the genericity of hyperbolic 3-manifolds, and the   counting of common perpendiculars between subsets in general metric spaces.

Let us first  introduce the abstract setup for counting double cosets. 
Let $G$ be a countable group equipped with a proper left invariant pseudo-metric $d_G$, so the ball of radius $r$ centered at the identity $1$ $$B_G(r)=\{g\in G: d_G(1, g)\le r\}$$  is finite. The function  $ \mathrm{gr}_G(r)=\sharp B_G(r)$ is called the \textit{growth function} of $G$ with respect to $d_G$. 

One of the main goals of this paper is to study the growth function of double cosets. Namely, for any two subgroups $H,K\leq G$, the \textit{double coset growth function} $\mathrm{gr}_{H,K}(r)$ counts  the number of double cosets $HgK$ intersecting the ball of radius $r$, i.e.
$$\mathrm{gr}_{H,K}(r)=\sharp B_{H,K}(r)$$ where $B_{H,K}(r)=\{HgK: d_G(HgK,1):=\min\limits_{h\in H,k\in K}d_G(hgk,1)<r\}$. If $H$ and $K$ are trivial, then $\mathrm{gr}_{H,K}(r)=\mathrm{gr}_G(r)$ and in general, we have $\mathrm{gr}_{H,K}(r)\leq \mathrm{gr}_G(r)$ for any $r\ge 0$. If $H,K$ are normal subgroups, it is an elementary observation in \cite{GK93} (recalled as Lemma \ref{normaldoublecoset}) that  $\mathrm{gr}_{H,K}(r)$ is exactly the growth function of the quotient group $G/HK$ endowed with the quotient metric. Due to this reason, we consider only non-normal subgroups in what follows.


In this paper, the proper left-invariant pseudo-metric on $G$   comes from a proper and isometric action  of $G$ on a proper geodesic metric space $(X,d)$. Fixing a basepoint $o\in X$ induces a pseudo-metric on $G$ by $d_G(g_1,g_2):=d(g_1o,g_2o)$ for any $g_1,g_2\in G$. We always assume that group $G$ is \textit{non-elementary}: it is neither virtually cyclic nor finite.

We say a subset $A\subseteq X$ is \textit{strongly contracting} if the shortest point projection $\pi_A(B)$ has uniformly bounded diameter for any ball $B$ disjoint from $A$. An element $g\in G$ of infinite order is called \textit{strongly contracting} if it acts by translation on a strongly contracting quasi-geodesic. We remark that several different notions of strongly contracting elements exist in the literature, for which  we refer the reader to \cite{ACGH} for  details.  Since this is the only notion used throughout the paper, we simply call contracting elements / subsets to be consistent with the ones called in \cite{Yang14,Yang19,Yang22}.

\subsection{Double coset growth for Morse subgroups}
As a warm-up, we first present our counting results for double cosets of a pair of Morse subgroups  before stating the more general result. The class of Morse subgroups has received much attention in recent years (\cite{Gen19,Tran}).

A subset $A\subseteq X$ is called \textit{Morse} if it is  quasi-convex with respect to all quasi-geodesics of any fixed parameters. A subgroup $H$ is called \textit{Morse} if it acts cocompactly on a Morse subset. It is well known that a contracting subset is Morse, so a contracting element generates a cyclic Morse subgroup.
In general, a Morse subgroup could be much more complicated.

We say that a group $G$  has \textit{purely exponential growth} if there exist $M_0,M_1>0,\omega>1$ so that
\[\forall r\ge 1:\; M_0\omega^r\leq\mathrm{gr}_G(r)\leq M_1\omega^r \]
for which  we shall write  $\mathrm{gr}_G(r)\asymp \omega^r$ for simplicity.
The value  $\omega=\omega_G$ is called the \textit{growth rate} (or the \textit{critical exponent}) of $G$ which can be defined a prior as follows
$$
\omega_G=\limsup_{n\to\infty} \frac{\log \mathrm{gr}_G(r)}{r}.$$

Our first result  is that the growth function of double cosets for any pair of Morse subgroups of infinite index is bounded below by the orbital growth function.

\begin{Theorem} \label{Thm:Growth-DoubleCosets}
Suppose that a non-elementary group $G$ acts properly on a proper geodesic metric space with a contracting element. Assume that $H$ and $K$ are Morse subgroups of infinite index. Then  
there exist $\delta, r_0>0$ so that for any $r>r_0$, $$\mathrm{gr}_{H,K}(r)\geq \delta \cdot\mathrm{gr}_G(r-r_0)$$
In particular, if $G$ has purely exponential growth, then $\mathrm{gr}_{H,K}(r)\asymp\mathrm{gr}_G(r) \asymp\omega_G^r$.
\end{Theorem}


By \cite[Theorem 7.2]{Coornaert93}, hyperbolic groups have purely exponential growth. Quasi-convex subgroups of hyperbolic groups are exactly  Morse subgroups, so Theorem \ref{Thm:Growth-DoubleCosets}    generalizes the result \cite[Theorem 2]{GitikRips20} about the growth of double cosets for quasi-convex subgroups in hyperbolic groups. If one of $H, K$ is trivial, we recover the recent result of \cite[Theorem 1.4]{Legaspi22}: indeed, the above inequality implies the coincidence of the left and right coset growth rate:    $\omega_{G/H}=\omega_{H\backslash G}=\omega_G$.

A class of \textit{statistically convex-cocompact} (SCC) actions was introduced in \cite{Yang19} as a generalization of convex-cocompact (thus any cocompact) actions in a statistical sense. This includes for examples the list of proper actions in Theorem \ref{Thm:Growth-DoubleCosets2} below. Such group actions are proved there to have purely exponential growth \cite[Theorem B (3)]{Yang19}. See \textsection\ref{Sec:SCCAction}   for precise definitions and relevant facts.  This is the main class of groups we shall be interested in later on.  

\begin{Corollary}
Suppose that a non-elementary group $G$ acts statistically convex-cocompactly on a geodesic metric space with a contracting element. 
Then for any two Morse subgroups $H,K\leq G$ of infinite index, $\mathrm{gr}_{H,K}(r)\asymp \omega_G^r$. 
\end{Corollary}

\subsection{Double coset growth for subgroups of second type}
We now extend Theorem \ref{Thm:Growth-DoubleCosets} to a much larger class of subgroups defined using boundary, which properly contains  Morse subgroups of infinite index. This is based on an axiomized notion of convergence boundary in \cite{Yang22}, which includes Gromov boundary of hyperbolic spaces, Bowditch boundary  (\cite{Bowditch12,GerPot13,GerPot16}) or Floyd boundary  (\cite{Floyd80,Gerasimov12}) of relatively hyperbolic groups, visual boundary of CAT(0) spaces  (\cite{Ballmann95}) and Thurston boundary of Teichm\"{u}ller spaces. This set of two axioms recalled below  enables us to introduce a good notion of limit sets for a non-elementary group which can be described by the following two alternative desired properties: \begin{enumerate}
    \item 
    The limit set is the minimal   invariant closed subset in the boundary.
    \item
    The limit set is the set of accumulation points of group orbits in the space. 
\end{enumerate}   

Before stating our general result, let us first describe  the  main application for the growth of double cosets in the following various classes of groups.

\begin{Theorem}\label{Thm:Growth-DoubleCosets2}
Suppose that a triple $(G, X, \partial X)$, where a finitely generated group $G$ acts properly on a proper length space $X$ compactified by a  boundary $\partial X$, is given by one of the following:
\begin{enumerate}
\item a hyperbolic group $G$ acts geometrically on a hyperbolic space  with Gromov boundary;
\item a relatively hyperbolic group $G$ acts on its Cayley graph with Bowditch boundary;
\item a finitely generated group $G$ acts on its Cayley graph with nontrivial Floyd boundary;
\item a  group $G$ acts  geometrically with a rank-1 element on a CAT(0) space  with visual boundary;
\item the mapping class group $G$ acts  on the Teichm\"{u}ller space $X$ with Thurston boundary.
\end{enumerate}
Let $H,K\leq G$ be any two subgroups whose limit sets are proper subsets in that of $G$. Then  $\mathrm{gr}_{H,K}(r)\asymp \omega_G^r$. 
\end{Theorem}

In the above items  (1) -- (3), the limit sets have the two desirable properties as above, which are well-known consequences of convergence group actions. See \cite{Floyd80,Ka,Gerasimov12,Bowditch12,GerPot13} for  more details. The item (4) follows from the work of Hamenstadt \cite{H09a}. The last one is more subtle, requiring a further explanation of the above-mentioned compactification with \textit{convergence property}.

A compact metric space $\overline{X}$ is a \textit{compactification} of $X$ if $X$ embeds into $\overline{X}$ as an open and dense subset, and $\partial X:=\overline{X}\setminus X$ is called the \textit{boundary}. We assume that the isometry group $\mathrm{Isom}(X)$ extends by homeomorphisms to  $\partial X$.  We equip $\partial X$ with an $\mathrm{Isom}(X)$-invariant partition $[\cdot]$. Denote   $[A]=\cup_{a\in A}[a]$ the $[\cdot]$-saturation of a subset $A$, which by definition is the union of all $[\cdot]$-classes over $A$. We say that a sequence of subsets $A_n$ in $X$ \textit{tends to} the $[\cdot]$-class  $[\xi]$ of a limit point $\xi\in \partial X$ if any unbounded convergent sequence of points $a_n\in A_n$ tends to a point in  $[\xi]$.  A sequence of subsets $A_n$ is called \textit{escaping} if $d(o,A_n)\rightarrow+\infty$ for some (and thus any) $o\in X$. 
The cone $\Omega_o(A)$ of a subset $A$ is the set of points $x\in X$ such that $[o,x]\cap A\ne\emptyset$.


\begin{Definition}\label{ConvBdryDefn}
The  compactification $\overline{X}$ has \textit{convergence property} with respect to the partition $[\cdot]$ if the following three assumptions hold:
\begin{enumerate}
\item[(A).]
For any contracting quasi-geodesic ray $\gamma$, $\gamma_n=\gamma$ tends to a closed $[\cdot]$-class denoted by $[\gamma^+]$ in $\partial X$. Furthermore, if $x_n\in X$ is a sequence of points with $\lim\limits_{n\rightarrow+\infty}d(o,\pi_{\gamma}(x_n))=+\infty$, then $x_n$ tends to  $[\gamma^+]$.
\item[(B).]
Let $\gamma_n$ be an escaping sequence of $C$-contracting quasi-geodesics for some given $C>0$. Then the sequence of  the cones $\Omega_o(\gamma_n)$ for any $o\in X$  has a subsequence tending to a $[\cdot]$-class $[\xi]$ for some limit point $\xi\in \partial X$.
\item [(C).] The set of \textit{non-pinched} points $\xi\in \partial X$ is non-empty:   for any two sequences $x_n\to [\xi]$ and $y_n\to [\xi]$, the sequence of geodesics $[x_n, y_n]$ is escaping.
\end{enumerate}
\end{Definition}
A partition is called \textit{trivial} if all $[\cdot]$-classes are singletons. On the opposite, any compactification has the convergence property for the coarsest partition by asserting $\partial X$ as one $[\cdot]$-class.
\begin{Remark}
In Assumption (A), we emphasize $[\gamma^+]$ is a closed subset in $\partial X$, even though the partition $[\cdot]$ is not assumed to  be a closed relation (so not all $[\cdot]$-classes are need to be closed subsets).  
\end{Remark}

For a subgroup $H\leq G$, let $\Lambda Ho$ denote the collection of accumulation points of $Ho$ in $\partial X$. Assumption (B) implies that the set $[\Lambda Ho]$ is independent of the choice of basepoint $o\in X$. We shall call $[\Lambda Ho]$ the \textit{limit set} of $H$. If  a non-elementary subgroup $H$ has a contracting element, then  $\Lambda Ho$ is the minimal $H$-invariant subset in $\partial X$ up to taking $[\cdot]$-closure (see Lemma \ref{Lem:MinLimitSet} for details). 

The compactifications in the first 4 items of Theorem \ref{Thm:Growth-DoubleCosets2} are equipped with the trivial partition. We now elaborate on the item (5) where the partition is non-trivial.

Let $\Sigma_g$ be a closed oriented surface  of genus $g\geq2$, and  denote by $\mathrm{Mod}(\Sigma_g)$  the group of isotopy classes of orientation-preserving homeomorphisms of $\Sigma_g$.  Let $\mathcal{T}(\Sigma_g)$ be the Teichm\"{u}ller space of $\Sigma_g$, i.e. the set of isotopy classes of marked complex structures on $\Sigma_g$,  equipped with the Teichm\"{u}ller metric $d_\mathcal{T}$. We consider the Thurston boundary the$\partial_{\mathrm{Th}}\mathcal{T}(\Sigma_g)\cong \mathcal{PML}(\Sigma_g)$, i.e., the projective  measured lamination space,  of $\mathcal{T}(\Sigma_g)$. 

It is observed in \cite[Thm 1.1]{Yang22} that Thurston boundary has convergence property  with respect to the Kaimanovich-Masur partition in  \cite{KaiMas96}. This partition restricted on the minimal foliations is exactly the zero intersection relation, which is however not true on the whole boundary $\mathcal{PML}(\Sigma_g)$. In particular, it restricts trivially on  uniquely ergodic points.  

By the work of Eskin-Mirzakhani-Rafi \cite{EskMirRaf19}, the proper action of $\mathrm{Mod}(\Sigma_g)=:G$    on $(\mathcal{T}(\Sigma_g),d_\mathcal{T})$ has purely exponential growth with growth rate $\omega_G=6g-6$, where all pseudo-Anosov elements are contracting by Minsky \cite{Minsky96}.

In \cite{MP89}, McCarthy-Papadopoulos studied the limit set of subgroups in  $\mathrm{Mod}(\Sigma_g)$ on the boundary $\mathcal{PML}(\Sigma_g)$. If $H\leq\mathrm{Mod}(\Sigma_g)$ is   \textit{sufficiently large} (i.e. containing two independent pseudo-Anosov elements), then the limit set $\Lambda_{\textrm{MP}} H\subseteq\mathcal{PML}(\Sigma_g)$ is defined as the closure of all fixed points of pseudo-Anosov elements in $H$. By Lemma \ref{Lem:MinLimitSet}, its relation with the above limit set $\Lambda Ho$ is the following:   $[\Lambda_{\textrm{MP}} H] =[\Lambda Ho]$. Moreover,     $[\Lambda Ho]$ is proper in $\mathcal{PML}$ if (and only if)  $[\Lambda_{\textrm{MP}} H]$ is so. 

We recapitulate the item (5) in Theorem \ref{Thm:Growth-DoubleCosets2} as follows.

\begin{Corollary} \label{DCG-Mod}
Suppose that $\mathrm{Mod}(\Sigma_g)=:G$ acts on $(\mathcal{T}(\Sigma_g),d_\mathcal{T})$.
Consider any two subgroups $H,K\leq G$ with proper limit sets $[\Lambda Ho],[\Lambda Ko]\neq \mathcal{PML}$, or   $\Lambda_{\textrm{MP}} H,\Lambda_{\textrm{MP}} K\neq \mathcal{PML}$ if $H, K$ are sufficiently large. Then  $\mathrm{gr}_{H,K}(r)\asymp \mathrm{gr}_{G}(r)\asymp e^{(6g-6)r} $.  
\end{Corollary}

To conclude our discussion, we mention the following general result.

Following the classical terminology in Kleinian groups, we say that a subgroup $H<G$ is of \textit{second type} if the limit set $[\Lambda Ho]$ is a proper subset of the whole limit set $[\Lambda Go]$; otherwise it is called of \textit{first type}. 
For a proper geodesic metric space, the horofunction compactification   has convergence property with respect to the finite difference partition by \cite[Thm 1.1]{Yang22}.

\begin{Theorem}\label{Thm:Growth-DoubleCosets-Boundary-SecType}
Suppose that $X$ is a proper geodesic metric space either with  \begin{enumerate}
    \item    the
horofunction  boundary $\partial_h X$, or with
    \item 
    a  boundary $\partial X$ with convergence property, where  non-pinched contracting elements exist in $G$ with the minimal fixed points (i.e.: their $[\cdot]$-classes are singleton and non-pinched).
\end{enumerate} Then for any two subgroups $H,K\leq G$ of second type, there exist $\delta,r_0>0$ so that $\mathrm{gr}_{H,K}(r)\geq \delta\cdot \mathrm{gr}_G(r-r_0)$ for all $r>r_0$.
In particular, if $G$ has purely exponential growth, then $\mathrm{gr}_{H,K}(r)\asymp \omega_G^r$. 
\end{Theorem}



\begin{Remark}
\begin{enumerate}
    \item 
For a horofunction compactification,   Morse subgroups of infinite index are of second type by Corollary \ref{Cor:Morse2ndtype}. Theorem \ref{Thm:Growth-DoubleCosets-Boundary-SecType} thus provides a strictly larger class of subgroups with the desired    double coset growth.
\item 
The list of all examples in Theorem \ref{Thm:Growth-DoubleCosets2} satisfy the second item: there are non-pinched contracting element with minimal fixed points. 
\end{enumerate}
\end{Remark}

The proof of Theorem \ref{Thm:Growth-DoubleCosets-Boundary-SecType} also implies the following combination result. 
\begin{Proposition}\label{freeproductcombProp}
Under the assumptions of Theorem \ref{Thm:Growth-DoubleCosets-Boundary-SecType}, there exist infinitely many $g\in G$ such that the subgroup $\langle H, gKg^{-1}\rangle\leq G$ is isomorphic to the free product $H\ast gKg^{-1}$.     
\end{Proposition}
As a corollary, we  recover the following specific cases which had already been known in $\mathrm{Mod}(\Sigma_g)$, while our proof uses only the  techniques before Masur-Minksy's curve complex machinery:
\begin{enumerate}
    \item 
    $H$, $K$ are finite subgroups (\cite[Corollary 9]{Min_2011});
    \item 
    $H,K$ are handlebody subgroups (\cite[Theorem 2]{OS16}).
\end{enumerate}

\subsection{Applications to generic 3-manifolds} \label{subsec:App-3Man}

Let $V$ be a 3-dimensional handlebody with the boundary surface $\Sigma_g$ for $g\ge 2$. Then the \textit{handlebody subgroup} $H<\mathrm{Mod}(\Sigma_g)$ consists of  the isotopy classes of orientation-preserving homeomorphisms of $\Sigma_g$  extending to the interior of $V$. Denote by $\mathcal {D}(V)$ the isotopy classes of essential simple closed curves of $\Sigma_{g}$ bounding discs in $V$. By  Masur \cite[Theorem 1.2]{Masur86}, the limit set of $\mathcal{D}(V)$ in $\mathcal{PML}(\Sigma_{g})$ is a unique closed, connected, $H$-invariant and minimal set with empty interior. By  \cite{MP89}, it agrees with the McCarthy-Papadopoulos limit set $\Lambda_{\textrm{MP}} H$ of the handlebody subgroup $H$. So $\Lambda_{\textrm{MP}} H\neq\Lambda\mathrm{Mod}(\Sigma_g)$. An immediate application of Corollary \ref{DCG-Mod} is that, the growth of double cosets
$\{H\varphi H:\varphi\in\mathrm{Mod}(\Sigma_g)\}$ is comparable with the growth of $\mathrm{Mod}(\Sigma_g)$.

For any $\varphi\in \mathrm{Mod}(\Sigma_g)$, gluing two copies $V_1, V_2$ of 
 $V$ along their boundary surfaces by $\varphi$, denoted by $V_1\cup_\varphi V_2$, gives a \textit{Heegaard splitting} of the resulted manifold $M_\varphi=V_1\cup_\varphi V_2$.
It is known that for any two  elements $\phi,\phi'$ in the same double coset $H\phi H=H\phi' H$, we obtain homeomorphic 3-manifolds $M_\phi\cong M_{\phi'}$. So a double coset corresponds to the same 3-manifold (but not vice versa). The main application is to study the genericity of  hyperbolic 3-manifolds $M_\varphi$. To make this precise, we shall introduce a geometric quantity on the collection of all orientable connected closed 3-manifolds with a genus $g$ Heegaard splitting,  up to homeomorphism: $$\Delta_{g}=\{M_{\varphi}\mid M_{\varphi}=V_1\cup_{\varphi} V_2, \varphi\in \mathrm{Mod}(\Sigma_{g})\}$$ This allows to  define a geometric complexity on  $\Delta _{g}$ using  {Teichm\"{u}ller metric} as follows. For any $M\in \Delta_g$, let $$D(M)=\{\varphi\in \mathrm{Mod}(\Sigma_g): M_\varphi\cong M\}$$ where $\cong $ denotes the homeomorphism.  Note that $D(M)$ is a union of these double cosets $H\varphi H$. For a fixed basepoint $o\in\mathcal{T}(\Sigma_g)$, the \textit{geometric complexity} $c(M)$ of $M$ is then defined to be the minimal translated distance of the gluing map $\phi\in D(M)$: 
$$
c(M):=\min\{d_{\mathcal {T}}(o, \varphi o): \varphi\in D(M)\}.
$$
As $\mathrm{Mod}(\Sigma_g)$ acts properly on $\mathcal{T}(\Sigma_g)$, the so-defined function $c: \Delta_{g}\to \mathbb R_{\ge 0}$ is a proper function.

We say that a sequence of real numbers $a_n$ converges to $a$ \textit{exponentially quickly} if $|a_n-a|\le c\varepsilon^n$ for some $c>0, \varepsilon\in (0,1)$ and for any $n\ge 1$.  A subset $\Gamma$ in $\Delta_{g}$ is said \textit{exponentially generic} with respect to geometric  complexity if the following converges exponentially quickly:
$$
\frac{\sharp\{M\in \Gamma: c(M)\le n\}}{\sharp\{M\in \Delta_g: c(M)\le n\}}\to 1, \text{ as } n\to \infty.
$$

Answering positively  a question of Thurston, Maher \cite{Maher} proved that a random Heegaard splitting gives rise to a hyperbolic 3-manifold $V_1 \cup_\phi V_2$, where $\phi\in \mathrm{Mod}(\Sigma_g)$ is sampled under a finite supported random walk. See Lubotzky-Maher-Wu \cite{LubMahWu16} and Maher-Schleimer \cite{MahSch21} for refined statements. In \cite[Question 6.4]{DHM15}, Maher asked  whether the proportion of orbit points  in the ball of radius $n$ in the Teichm\"{u}ller metric which
give rise to hyperbolic manifolds  when used as Heegaard splitting gluing maps tends to 1 as $n\to \infty$.  Our  following result   answers his   conjecture in the positive. We emphasize that however,  our   model counts the homeomorphic types of  3-manifolds  rather than the isotopy classes of Heegaard splittings, so it gives a more accurate picture of random 3-manifolds.  

\begin{Theorem} \label{Thm:HypExpGen}
Let $$\Gamma_{g}=\{M_\varphi :M_\varphi~\text{is hyperbolic for some}~\varphi \in \mathrm{Mod}(\Sigma_{g})\}.$$ Then $\Gamma_{g}$ is exponentially generic in $\Delta_g$ with respect to the geometric complexity $c$. Moreover, we have 
\begin{align}\label{mfdBallEQ}
\sharp\{M\in \Delta_g: c(M)\le n\} \asymp \mathrm{e}^{(6g-6)n}    
\end{align}
where the implicit constant might depend on the basepoint $o\in\mathcal{T}(\Sigma_g)$.
\end{Theorem}


 It is worth noting that in \cite{QZG,ZQZ}, Qiu, Guo, Zhang and the third author proved that given any $g\geq 2$, there are infinitely many non-homeomorphic hyperbolic 3-manifolds, which admit genus $g$ Heeegaard splittings. So our result is a refined version of their result. The crucial ingredient to Theorem \ref{Thm:HypExpGen} is that the Heegaard/Hempel distance $d_H(V\cup_{\varphi} V)$, defined by Hempel \cite{Hempel01}, has positive drift with exponential decay. 

\begin{Proposition} \label{Prop:Hemple-Nondeg}
There exists   $\kappa\in(0,1)$ so that the following sets
$$\mathcal{H}=\{H\varphi H:d_H(V_1\cup_{\varphi} V_2)\geq\kappa \cdot d_\mathcal{T}(o,H\varphi Ho)\}$$
and $$\mathcal{H}'=\{M_{\varphi}:d_H(V_1\cup_{\varphi} V_2)\geq\kappa\cdot  d_\mathcal{T}(o,H\varphi Ho)\}$$
are exponentially generic in $\{H\varphi H:\varphi\in\mathrm{Mod}(\Sigma_g)\}$, and  in $\Delta_g$ respectively. 
\end{Proposition}

\begin{Remark}
In fact, we can view 
$$d_\mathcal{T}(o,H\varphi Ho)=:\|H\varphi H\|_1,~~d_H(V_1\cup_{\varphi} V_2)=:\|H\varphi H\|_2$$
as two semi-norms defined on the collection of double cosets $\mathbb{D}(H,H)=\{H\varphi H:\varphi\in \mathrm{Mod}(\Sigma_{g})\}$. So $\mathcal{H}$ is just the following set
\[\{\Upsilon\in\mathbb{D}(H,H):\|\Upsilon\|_2\geq\kappa\|\Upsilon\|_1\}.\]
A subset $\mathbb{A}\subseteq\mathbb{D}(H,H)$ is exponentially generic if the ratio $\frac{\sharp(\mathbb{A}\cap\{\Upsilon\in\mathbb{D}(H,H):\|\Upsilon\|_1\leq r\})}{\sharp\{\Upsilon\in\mathbb{D}(H,H):\|\Upsilon\|_1\leq r\}}\rightarrow1$ exponentially quickly.
Moreover, if $\overline{\Pi}$ is the surjective map $H\varphi H\in\mathbb{D}(H,H)\mapsto M_\varphi\in\Delta_g$ then $\mathcal{H}'$ is described as the following set
\[\{\overline{\Upsilon}\in\Delta_g:\|\Upsilon\|_2\geq\kappa\|\Upsilon\|_1~\text{for each}~\Upsilon\in\overline{\Pi}^{-1}(\overline{\Upsilon})\}.\]

\end{Remark}

\subsection{Counting common perpendiculars and further questions } 

By definition, the double coset growth $\mathrm{gr}_{H, K}(r)$  counts  the number of  subsets  $\phi Ko$ ($\phi\in G$) within a distance at most $r$ to $Ho$. If $\pi: X \to X/G$ is the covering map (e.g. when the action is free), then $\mathrm{gr}_{H, K}(r)$ counts  common perpendiculars between subsets $A=\pi(Ho)$ and $B=\pi(Ko)$. Considering the cover $Y=X/H \to X/G$ associated to $H$, $\mathrm{gr}_{H, K}(r)$ amounts to counting how many subsets of the form $\pi(\phi Ko)$ in the $r$-ball centered at $Ho\in Y$.  

As earlier-mentioned, this geometric scenario has been investigated in Parkkonen-Paulin's work    \cite{PP17}. They obtained the precise asymptotic on the number of common perpendiculars between    two properly immersed closed convex sub-manifolds $A, B$ in a negatively pinched Riemannian manifold $M$. If $A,B$ are compact, then their fundamental subgroups $H=\pi_1(A), K=\pi_1(B)$ are convex-cocompact  in $\pi_1(M)$. In this case, the coarse double coset growth for $H, K$ by Theorem \ref{Thm:Growth-DoubleCosets-Boundary-SecType} follows from \cite{PP17}.

It must go beyond the scope of this paper to derive a precise formula for the double coset growth function. However, a    precise asymptotic estimation is expected in various classes of groups as demonstrated in \cite{PP17}.  In view of the formula (\ref{mfdBallEQ}), we would like to ask the following question.

\begin{Question}
Let $H\subset\mathrm{Mod}(\Sigma_g)$ be the handlebody group. Does there exist a constant $C$ so that the following holds:
$$
\forall r>0:\; \mathrm{gr}_{H,H}(r) \sim C \cdot \mathrm{e}^{(6g-6)r} 
$$
for any basepoint $o\in \mathcal{T}(\Sigma_g)$?
\end{Question}
\begin{Remark}
We remark that the lattice counting for Teichm\"{u}ller metric (i.e.  $H$ is trivial) has the above      precise form, where   the constant $C$ actually does not depend on the basepoint $o$ (\cite{AthBufEskMir12}).      
\end{Remark}

The paper is  organized as follows.  In Section \textsection\ref{SecPrelim}, we recall the basic results about contracting elements, statistically convex-cocompact actions, and   the convergence compactification, moreover we derive several preparatory results for counting double cosets. Sections \textsection\ref{SecMorseGrowth} and \textsection\ref{SecSecondType} are devoted to the proofs of Theorem \ref{Thm:Growth-DoubleCosets} and Theorem \ref{Thm:Growth-DoubleCosets-Boundary-SecType}, where the latter is modeled on the former with necessary changes presented in \textsection\ref{SecPrelim}.   The final Section \textsection\ref{SecApp3Mfds} presents the application, Theorem \ref{Thm:HypExpGen}, to the genericity of hyperbolic 3-manifolds.

\subsection*{Acknowledgments}
We are grateful to Prof. Ruifeng Qiu for helpful remarks and Prof. Fr\'ed\'eric Paulin for suggesting references. Many thanks to Dr. Xiaolong Han for many corrections. S. H. is supported by the Special Research Assistant Project at Chinese Academy of Sciences (E2559303). W. Y. is supported by National Key R \& D Program of China (SQ2020YFA070059).
Y.Z. is supported by NSFC No.12131009 and in part by Science and Technology Commission of Shanghai Municipality (No. 22DZ2229014).

\section{Preliminaries}\label{SecPrelim}
\subsection{Contracting element} \label{Sec:Contracting}

Suppose that $(X,d)$ is a proper geodesic metric space. Then for any $x,y\in X$, we denote by $[x,y]$ a choice of geodesic in $X$ from $x$ to $y$. The (closed) $r$-neighborhood of a subset $A\subseteq X$ for $r\ge 0$ is denoted by $N_r(A)$. 

Given a point $x\in X$ and a closed subset $A\subseteq X$, let $\pi_A(x)$ be the shortest point projection of $x$ to $A$, i.e., $\pi_A(x)$ is the set of point $a\in A$ such that $d(x,a)=d(x,A)$. The \textit{projection} of a subset $Y\subseteq X$ to $A$ is then $\pi_A(Y):=\mathop{\cup}\limits_{y\in Y}\pi_A(y)$.

Suppose that a group $G$ acts properly on $(X,d)$. Then for a fixed basepoint $o\in G$ as before, we similarly denoted by $\mathcal{N}_r(A)=\{g\in G:d_G(g,A)\leq r\}$ the $r$-neighborhood of a subset $A\subseteq G$ w.r.t. the pseudo-metric $d_G$.

Unless otherwise specified, the path $\gamma: I\subseteq [-\infty,+\infty]\to X$   is     equipped with arc-length parametrization,  and $\ell(\gamma)\in [0, \infty]$ denotes the length of $\gamma$.  
A path $\gamma$ is called a \textit{$\lambda$-quasi-geodesic} for $\lambda\geq1$ if for any rectifiable subpath $[\gamma(s),\gamma(t)]_\gamma:=\gamma|_{[a,b]}\subseteq\gamma$, we have
\[\ell([\gamma(s),\gamma(t)]_\gamma) \leq \lambda d(\gamma(s),\gamma(t))+\lambda.\]
We refer to $\gamma_-,\gamma_+$ as the initial and terminal endpoints of $\gamma$ respectively. And we define the path labelled by the product $g_1g_2\cdots g_n\in G$ as the concatenation of geodesic segments $[o,g_1o]\cdot g_1[o,g_2o]\cdot\cdots (g_1\cdots g_{n-1})[o,g_no].$

\begin{Definition}
A subset $A\subseteq X$ is called $C$-\textit{contracting} for $C\geq0$ if for any geodesic $\gamma$ with $d(\gamma,A)\geq C$, we have   $\mathrm{diam}(\pi_A(\gamma))\leq C$.

An element $g\in G$ is called \textit{contracting} if for some basepoint $o\in X$,     the $\langle g\rangle$-invariant concatenated path $\mathop{\cup}\limits_{n\in\mathbb{Z}}g^n([o,go])$
is a $C$-contracting quasi-geodesic for some $C>0$.
\end{Definition}

\begin{Definition} \label{Def:Morse}
A subset $A\subseteq X$ is \textit{$\eta$-Morse} for a function $\eta:\mathbb{R}_{\geq0}\rightarrow\mathbb{R}_{\geq0}$ if every $\lambda$-quasi-geodesic with two endpoints in $A$ is contained in the $\eta(\lambda)$-neighborhood of $A$.

A subgroup $H\leq G$ is \textit{Morse} if  the subset $H o$ for a basepoint $o\in X$ is $\eta$-Morse for some function $\eta:\mathbb{R}_{\geq0}\rightarrow\mathbb{R}_{\geq0}$. 
\end{Definition}

We state some basic properties of contracting subset, whose proof is left to the interested reader.

\begin{Lemma}\label{Lem:BigThree}
\begin{enumerate}
\item 
Let $A$ be a $C$-contracting subset. If $\mathrm{diam}(\pi_A(x)\cup\pi_A(y))>C$ then 
$$d(\pi_A(x),[x,y]),d(\pi_A(y),[x,y])\le 2C.$$

\item
If a subset $A$ has  a Hausdorff distance at most $D$ to a $C$-contracting subset $B$, then $A$ is $C'$-contracting for some $C'$ depending on $C$ and $D$. 

\item
If $\alpha$ is a $C$-contracting $\lambda$-quasi-geodesic, then any subpath of $\alpha$ and any geodesic with endpoints in $\alpha$ are $C'$-contracting for some $C'=C'(\lambda, C)$.
\end{enumerate}
\end{Lemma} 
Hence, the definition of a contracting element is independent of the choice of basepoint.   

By \cite[Lemma 2.11]{Yang19}, each contracting element $g$ is contained in a maximal elementary  group  $E(g)$ defined as follows:
\[E(g)=\{h\in G:\exists r>0,h\langle g\rangle o\subseteq N_r(\langle g\rangle o) ~\text{and}~
\langle g\rangle o\subseteq N_r(h\langle g\rangle o)\},\]
and the index $[E(g):\langle g\rangle]$ is finite.
That is to say, $E(g)$ is a virtually cyclic subgroup, and any such group containing $g$ is a subgroup of $E(g)$. The extended-defined quasi-geodesic denoted by $\mathrm{Ax}(f):=E(f)o$ is called the \textit{axis} of $g$ (depending on $o$).
 


A collection $\mathbb A$ of  subsets has \textit{bounded intersection} property if for any $R>0$ and $A\ne A'\in\mathbb A$, the diameter of $N_R(A)\cap N_R(A')$ is bounded by a constant depending only on $R$. If $\mathbb A$ consists of  $C$-contracting subsets, this is equivalent to the \textit{bounded projection} property: for any $A\ne A'\in\mathbb A$, $\pi_A(A')$ has diameter bounded by a uniform constant.

Two contracting elements $g,h\in G$ are called \textit{independent} if the collection $\{f\mathrm{Ax}(g), f\mathrm{Ax}(h): f\in G\}$ of their axis and translations has bounded intersection property. In algebraic terms, it amounts to saying that $g\notin fE(h)f^{-1}$ for any $f\in G$; or equivalently, $E(g)\neq fE(h)f^{-1}$ for any $f\in G$. A non-elementary group with a contracting element actually contains infinitely many pairwise independent contracting elements. (See \cite[Lemma 2.12]{Yang19}) 

Let $\mathbb A$ be a contracting system with bounded intersection property. The following notion of an admissible path allows     to construct   a quasi-geodesic  by concatenating geodesics via $\mathbb A$.
\begin{Definition}[Admissible Path]\label{AdmDef} Given $L,\tau\geq0$, a path $\gamma$ is called $(L,\tau)$-\textit{admissible} in $X$, if $\gamma$ is a concatenation of geodesics $p_0q_1p_1\cdots q_np_n$ $(n\in\mathbb{N})$, where the two endpoints of each $p_i$ lie in some $A_i\in \mathbb{A}$, and   the following   \textit{Long Local} and \textit{Bounded Projection} properties hold:
\begin{enumerate}
\item[(LL)] Each $p_i$  for $1\le i< n$ has length bigger than $L$, and  $p_0,p_n$ could be trivial;
\item[(BP)] For each $1\le i< n$, $A_i\ne A_{i+1}$ and $\max\{\mathrm{diam}(\pi_{A_i}(q_i)),\mathrm{diam}(\pi_{A_i}(q_{i+1}))\}\leq\tau$, where $q_0:=\gamma_-$ and $q_{n+1}:=\gamma_+$ by convention.
\end{enumerate} 
The sets $A_1,A_2,\cdots,A_n$ are referred to as the contracting subsets associated to the admissible path.
\end{Definition}
\begin{Remark}\label{ConcatenationAdmPath}
The paths $q_i$ could be allowed to be trivial, so (BP) condition is automatically satisfied. It will be useful to note that admissible paths could be concatenated as follows: Let $p_0q_1p_1\cdots q_np_n$ and $p_0'q_1'p_1'\cdots q_n'p_n'$ be $(L,\tau)$-admissible paths. If $p_n=p_0'$ has length bigger than $L$, then the concatenation $(p_0q_1p_1\cdots q_np_n)\cdot (q_1'p_1'\cdots q_n'p_n')$ has a natural $(L,\tau)$-admissible structure.  
\end{Remark}

We say that the admissible path $\gamma$ \textit{$\epsilon$-fellow travels} a geodesic $\alpha$ if $\alpha$ has the same endpoints as $\gamma$ and $d((p_i)_-, \alpha)\le \epsilon,d((p_i)_+, \alpha)\le \epsilon$ for each $p_i$.

The basic fact is that a ``long" admissible path is a quasi-geodesic.

\begin{Proposition}\label{pro:quasi-geo}\cite[Corollary 3.2]{Yang14}
Let $C$ be the contraction constant of $\mathbb{A}$. For any $\tau>0$, there are constants $L=L(C,\tau)>0,\Lambda=\Lambda(C,\tau), \epsilon=\Lambda(C,\tau)>0$ such that  any $(L,\tau)$-admissible path is a $\Lambda$-quasi-geodesic which $\epsilon$-fellow travels any geodesic with the same endpoints.
\end{Proposition}


Let $F$ be a set of three pairwise independent contracting elements. 
Denote by $\mathbb{A}=\{g\mathrm{Ax}(f):g\in G,f\in F\}$ the  collection of their axis and translations, i.e. $C$-contracting subsets for some $C>0$. Then $\mathbb{A}$ is a contracting system with bounded intersection property (see \cite[Lemma 2.12]{Yang19}).  The following lemma \cite[Lemma 2.14]{Yang19} is useful to build admissible quasi-geodesics. 
\begin{Lemma} \label{Lem:Extension}
There exist $\tau_0>0$ depending on $\mathbb{A}$ and $L=L(C,\tau)>0$ for any $\tau>\tau_0$, so that  the following property holds.

Choose any set $F=\{f_1,f_2,f_3\}$ with $f_i\in E(f_i)$ and $|f_i|>L$. For any two elements $g_1,g_2\in G$, there exists     $f\in F$ so that the path labelled by $g_1fg_2$ is an $(L,\tau)$-admissible path, where the associated contracting set is given by $g_1\ax(f)$.

\end{Lemma}



The following result proved in \cite[Proposition 2.9 and Lemma 2.14]{Yang19} gives a procedure to construct contracting elements.

\begin{Lemma} \label{Lem:Contracting}
Suppose that a group $G$ acts properly on $(X,d)$ with a contracting element. Then there exist a set $F\subseteq G$ of three contracting elements and $\lambda>0$ with the following property.

For any $g\in G$, there exists $f\in F$ such that the bi-infinite concatenated path $\mathop{\cup}\limits_{n\in\mathbb{Z}}(gf)^n([o,gfo])$   is a contracting $\lambda$-quasi-geodesic. In particular, $gf$ is  a contracting element.
\end{Lemma}

Following   \cite[Definition 4.1]{Yang19}, a geodesic $\alpha\subseteq X$  contains an $(\epsilon,f)$-\textit{barrier} $t\in G$ if the following holds   $$d(to,\alpha)\leq\epsilon,\; d(tfo,\alpha)\leq\epsilon.$$ Otherwise, $\alpha$ is called $(\epsilon,f)$-\textit{barrier-free}.
An element $g\in G$ is called $(\epsilon,M,f)$-\textit{barrier-free} if there exists a geodesic segment from a point in $B(o,M)$ to a point in $B(go, M)$ which is $(\epsilon,f)$-{barrier-free}.

\begin{Lemma}\label{SmallProj4Barrier}
Let $f$ be a contracting element with a $C$-contracting axis $\mathrm{Ax}(f)$. Then there exist $\epsilon=\epsilon(C,f), \tau=\tau(f)$ so that for a geodesic segment $\alpha$ with $\mathrm{diam} (\pi_{\mathrm{Ax}(f)}(\alpha))> \tau$, the $\alpha$ contains an $(\epsilon,f)$-barrier.      
\end{Lemma}
\begin{proof}
Let $x\in \pi_{\mathrm{Ax}(f)}(\alpha_-)$ and $y\in \pi_{\mathrm{Ax}(f)}(\alpha_+)$ so that   $d(x,y)\geq\mathrm{diam} (\pi_{\mathrm{Ax}(f)}(\alpha))-6C$. As $\mathrm{Ax}(f)$ is $C$-contracting, we have that if $d(x,y)> C$, then $d(x, \alpha), d(y, \alpha)\le 2C$. Recall that a geodesic with
endpoints in a contracting quasi-geodesic is also contracting (see Lemma \ref{Lem:BigThree}). The  geodesic $[x,y]$ is contracting. Note that $\cup_{n\in \mathbb Z} f^n[o,fo]$ is a quasi-geodesic, which is contained in a fixed neighborhood of $\mathrm{Ax}(f)$ by Morse property of contracting set. Thus, if $d(x,y)\ge \tau\gg d(o,fo)$ is large enough relative to $d(o,fo)$, there exists a constant $C_0$ such that the $C_0$-neighborhood of $[x,y]$ contains a segment labeled by $f$. That is to say, $[x,y]$  contains a $(C_0,f)$-barrier.  Choose $u,v\in \alpha$ so that $d(x,u), d(y,v)\le 2C$.  By Morse property again, $[x,y]$ is contained in a $C_1$-neighborhood of $[u,v]$ for some $C_1=C_1(C)$. Thus, $\alpha$ contains an $(\epsilon,f)$-barrier, where  $\epsilon:=C_0+C_1$. We complete the proof.
\end{proof}

\begin{Lemma} \label{Lem:ConEle-H}
Assume that $H$ is a Morse subgroup  of infinite index. Then for any $\epsilon\geq0$, there exists a contracting element $g\in G$ so that  for any element $h\in H$, $[o,ho]$ contains no $(\epsilon,  g)$-barrier. 
\end{Lemma}
\begin{proof}
By Lemma \ref{Lem:Contracting}, for any $g_0\in G$ there exists      $f\in F$ so that $g:=g_0f$ is contracting.  Fix any $\epsilon>0$ and set   $\epsilon_0=\max\{d(o,fo): f\in F\}+\epsilon$.

As $H$ is of infinite index in $G$, it holds that $G\setminus S\cdot H\cdot S$ is infinite for any finite set $S$. See the proof of \cite[Theorem 4.8]{Yang19} for details. 

Now, assume that $Ho$ is $\eta$-Morse for some function $\eta$.
By the proper action, the set $S:=\{g\in G: d(o,go)\le \epsilon_0+\eta(1)\}$ is finite.  Choose any element $g_0\in G\setminus S\cdot H\cdot S$. We claim that for any  element $h\in H$, $[o,ho]$ contains  
 no $(\epsilon_0, g_0)$-barrier. Indeed, if not, there exists $t\in G$ such that $d(to, [o,ho]),d(tg_0o, [o,ho])\le \epsilon_0$. As $[o, ho]\subseteq N_{\eta(1)}(Ho)$, we have $to, tg_0o\in N_{\epsilon_0+\eta(1)}(Ho)$ and  then $s_1=t^{-1}h_1, s_2=h_2^{-1}(tg_0)\in S$ for some $h_1, h_2\in H$. This implies  $g_0=s_1(h_1^{-1}h_2)s_2\in S\cdot H\cdot S$: we get a contradiction. Our claim thus follows.

We now prove that $[o,ho]$ contains no $(\epsilon, g)$-barrier, where $g=g_0f$ with some $g_0\in G\setminus S\cdot H\cdot S$. Suppose to the contrary that there exists $t\in G$ such that $d(to, [o,ho]), d(tgo, [o,ho])\le \epsilon$. As $f\in F$, we have $d(to, [o,ho]), d(tg_0o, [o,ho])\le \epsilon_0$: $[o,ho]$ contains $(\epsilon_0, g_0)$-barrier. This contradicts the above choice of $g_0$. The proof is complete.
\end{proof}

As a corollary, we obtain the following       key estimate.

\begin{Lemma}
\label{Lem:BddProj}
Let $H, K$ be  Morse subgroups of infinite index. Then there exist a contracting element $g\in G$ and $\tau>0$ so that for any $h\in H\cup K$, the shortest point projection  to $\mathrm{Ax}(g)$   is uniformly bounded:
$$\mathrm{diam}(\pi_{\mathrm{Ax}(g)}([o,ho]))\leq \tau$$
and $\langle g\rangle \cap (H\cup K)=\{1\}$.
\end{Lemma} 
\begin{proof}
As $G\setminus S\cdot (H\cup K)\cdot S$ is infinite for any finite set $S$,  the proof of Lemma \ref{Lem:ConEle-H} gives the same contracting element for $H$ and $K$. The uniform bounded projection follows by Lemma \ref{SmallProj4Barrier}. This also implies $\langle g\rangle \cap (H\cup K)=\{1\}$. Otherwise,  if some power of $g$ is contained in $H$ (or $K$), then $E(g)o$ has infinite intersection with $Ho$ (or $Ko$). This contradicts the uniform bounded projection as above.    
\end{proof}

\subsection{Counting double cosets}
Recall that a group $G$ acts properly on a proper metric space $(X,d)$.  A subset $A\subseteq G$ is called \textit{exponentially generic} if the proportion of $A$ in the ball $B_G(r)$ tends to 1 exponentially quickly: there exists $c>0, \varepsilon\in (0,1)$ so that 
$$\forall r>0:\;\left|\frac{\sharp(A\cap B_{G}(r))}{\sharp B_{G}(r)}-1\right|\le c\varepsilon^r.$$   On the opposite, a subset $B\subseteq G$ is called \textit{exponentially negligible} if its complement   is exponentially generic. 

Let $H,K<G$ be any two subgroups.  Consider  the following surjective map \begin{align*}
\Pi: &G\longrightarrow \mathbb{D}(H,K):=\{HgK: g\in G\}\\
&g\longmapsto HgK.
\end{align*} 
Then the preimage under $\Pi$ of $B_{H,K}(r)$ is exactly $B_{G}(r)$.
We first observe that the image of an exponentially generic set of $G$ under $\Pi$ is also exponentially generic in $\{HgK:g\in G\}$.

\begin{Lemma} \label{Lem:ExpGen}
Suppose that $G$ acts properly on $(X,d)$ and  $H,K<G$ are subgroups so that $\mathrm{gr}_{H,K}(r)\geq\delta\cdot \mathrm{gr}_G(r)$ for all     $r\ge 1$ and some $\delta>0$.
If a subset $A\subseteq G$ is exponentially generic in $G$, then the image $\Pi(A)$ is exponentially generic in $\mathbb{D}(H,K)$.
\end{Lemma}

\begin{proof}
Let $\tilde A:=\Pi^{-1}(\Pi(A))=\cup_{a\in A}HaK$ be the collection of double cosets supported by $A$ and $\tilde A^{c}=G\setminus \tilde A$ be its complement in $G$. Since $A\subseteq G$ is exponentially generic, so does $\tilde A$, hence there exists $\varepsilon\in(0,1)$ and $c>0$ so that $$\sharp(\tilde A^{c}\cap B_{G}(r))\leq c\varepsilon^r \cdot \sharp B_{G}(r)$$  By definition, $B_{H,K}(r)$ consists of double cosets $HgK$ so that $|g_0|\le r$ for some $g_0\in HgK$, which implies $B_{H,K}(r)=\Pi(B_{G}(r))$,
therefore
$$\sharp (\Pi(\tilde A^{c})\cap B_{H,K}(r))=\sharp \Pi(\tilde A^{c}\cap B_{G}(r))\leq\sharp(\tilde A^{c}\cap B_{G}(r)).$$
By assumption, $\sharp B_{H,K}(r)\ge\delta \sharp B_G(r)$. Consequently, $$\frac{\sharp(\Pi(\tilde A^{c})\cap B_{H,K}(r))}{\sharp B_{H,K}(r)}\leq \frac{c\varepsilon^r \cdot \sharp B_{G}(r)}{\delta \sharp B_G(r)}\le \frac{c}{\delta}\varepsilon^r.$$ The result follows by noticing that $\Pi(A)=\Pi(\tilde A)$ and $\Pi(G)\setminus \Pi(\tilde A)=\Pi(\tilde A^{c})$.
\end{proof}

We now consider the case that $H,K$ are normal subgroups of $G$. If $G$ is equipped with the proper left-invariant pseudo-metric $d_G$, then it induces a  proper left-invariant pseudo-metric $\bar d_\Gamma$ on the quotient group  $\Gamma:=G/HK$  which is defined as:
$$
\bar d_\Gamma(gHK,g'HK):=\inf\{d_G(1, g^{-1}hkg'): hk\in HK\}$$
The following  lemma is elementary, whose proof is straightforward by unravelling definitions. 
\begin{Lemma}\label{normaldoublecoset}
The double coset growth function $\mathrm{gr}_{H,K}(r)$ of $H,K$ is the same as the growth function $\mathrm{gr}_{\Gamma}(r)$ of $\Gamma$ with respect to the pseudo-metric $\bar d_\Gamma$.  
\end{Lemma}
\subsection{Statistically convex-cocompact actions}\label{Sec:SCCAction}
Given constants $0\leq M_1\leq M_2$, let $\mathcal{O}_{M_1,M_2}$ be the set of element $g\in G$ such that there exists some geodesic $\gamma$ between $N_{M_2}(o)$ and $N_{M_2}(go)$ with the property that the interior of $\gamma$ lies outside $N_{M_1}(Go)$.

\begin{Definition}[SCC Action]\label{SCCDefn}
If there exist positive constants $M_1,M_2>0$ such that $\omega({\mathcal{O}_{M_1,M_2}})<\omega_G<\infty$, then the proper action of $G$ on $Y$ is called \textit{statistically convex-cocompact} (SCC).
\end{Definition}
\begin{Remark}
The idea to define the set $\mathcal{O}_{M_1,M_2}$ is to look at the action of  the fundamental group of a finite volume Hadamard manifold on its universal cover. It is then easy to see that for appropriate constants $M_1, M_2>0$, the set $\mathcal{O}_{M_1,M_2}$ coincides with  the union of cusp subgroups up to a  finite Hausdorff distance. The assumption in SCC actions was called a \textit{parabolic gap condition} by Dal'bo, Otal and Peign\'{e} in \cite{DalOtaPei00}.  
\end{Remark}

Given $\epsilon,M>0$ and  any $f\in G$, let $\mathcal{V}_{\epsilon,M,f}$ be the collection of all $(\epsilon,M,f)$-barrier-free elements of $G$. 

Furthermore, let  $\mathcal{V}_{\epsilon,M,f}^{\theta,L}$ be the set of all elements $g \in G$ with the following properties:   
\begin{enumerate}
    \item 
    there exists a set $\mathbb{K}_0$ of disjoint connected subintervals $\alpha\subseteq[o,g o]$ with endpoints $\alpha_-,\alpha_+\in N_M(Go)$ so that every $\alpha$ is $(\epsilon,f)$-barrier-free.
    \item
    If   $\mathbb{K}=\{\alpha\in\mathbb{K}_0:\ell(\alpha)\geq L\}$  consists of those intervals in $\mathbb{K}_0$ with length at least $L$, then $$\sum_{\alpha\in\mathbb{K}}\ell(\alpha)\geq\theta d(o,go)$$
\end{enumerate}

The following results will be crucially used in the next sections.

\begin{Proposition}\label{pro:growthtight}\cite[Theorem B, C]{Yang19}, \cite[Theorem 5.2]{GekYang22}
Assume that a non-elementary group $G$ admits a SCC action on a proper geodesic space $(Y,d)$ with a contracting element. Let $M_0=M_1=M_2$ be the constant in the definition of SCC action. Then
\begin{enumerate}
\item $G$ has purely exponential growth.
\item  For any $M>M_0$, there exists $\epsilon=\epsilon(M)>0$ such that $\mathcal{V}_{\epsilon,M,f}$ is exponentially negligible for any $f\in G$.
\item
For any $0<\theta\leq 1$ there exists $L=L(\theta)>0$, so that the set $\mathcal{V}_{\epsilon,M,f}^{\theta,L}$ is exponentially negligible for any $f\in G$.
In other words, the complement subset $$\mathcal{W}_{\epsilon,M,f}^{\theta,L}:=G\setminus\mathcal{V}_{\epsilon,M,f}^{\theta,L}$$ is exponentially generic.
\end{enumerate}
\end{Proposition}

\subsection{Horofunction boundary}
Fix a basepoint $o\in X$. For  each $y \in  X$, we define a Lipschitz map $b_y:  X\to X$     by $$\forall x\in X:\quad b_y(x)=d(x, y)-d(o,y).$$ This   family of $1$-Lipschitz functions sits in the set of continuous functions on $ X$ vanishing at $o$.  Endowed  with the compact-open topology, the  Arzela-Ascoli Lemma implies that the closure  of $\{b^o_y: y\in  X\}$  gives a compactification of $X$.  The complement, denoted by $\partial_h X$, of $X$ is called  the \textit{horofunction boundary}. 

A \textit{Buseman cocycle} $B_\xi:  X\times X \to \mathbb R$ at $\xi\in \partial_h X$ (independent of $o$) is given by $$\forall x_1, x_2\in  X: \quad B_\xi(x_1, x_2)=b_\xi^o(x_1)-b_\xi^o(x_2).$$

The topological type of horofunction boundary is independent of  the choice of   basepoints. Every isometry $\phi$ of $X$ induces a homeomorphism on $\partial_h X$:
$$
\forall y\in X:\quad\phi(\xi)(y):=b_\xi(\phi^{-1}(y))-b_\xi(\phi^{-1}(o)).
$$
According to the context, both $\xi$ and $b_\xi$ are used to denote  the boundary points.

\paragraph{\textbf{Finite difference relation}.}
Two horofunctions $b_\xi, b_\eta$ have   \textit{$K$-finite difference} for $K\ge 0$ if the $L^\infty$-norm of their difference is $K$-bounded: $$\|b_\xi-b_\eta\|_\infty\le K.$$ 
The   \textit{locus} denoted by $[b_\xi]$  of     $b_\xi$ consists of  horofunctions $b_\eta$ so that $b_\xi, b_\eta$ have   $K$-finite difference for some $K$ depending on $\eta$.  This  defines a \textit{finite difference equivalence relation} $[\cdot]$ on $\partial_h X$. The \textit{locus} $[\Lambda]$ of a subset $\Lambda\subseteq \partial_h X$ is the union of loci of all points in $\Lambda$. We say that $\Lambda$ is \textit{saturated} if $[\Lambda]=\Lambda$. 

The finite difference relation might not be a closed relation in general. That is to say, if $[x_n]=[y_n]$, then the accumulation points of $x_n$ and $y_n$ might not be in the same $[\cdot]$-class. This, however, can be circumvented in a certain sense, if $[x_n]$ has    $K$-finite difference with a fixed $K>0$. 

\begin{Lemma}\label{CloseRelationOverOrbit}
Let $[x_n]$ be a sequence of $[\cdot]$-classes in the horofunction boundary with $K$-finite difference for a fixed $K>0$. Then any two accumulation points of the sets $[x_n]$ have  $K$-finite difference.       
\end{Lemma}
If $g$ is a contracting element, then the fixed point $[g^\pm]$ has  $K$-finite difference for some $K>0$ depending on the contracting constant of $\ax(g)$. Moreover, any translate of $[g^\pm]$ has $K$-finite difference as well. Therefore, the orbit of a fixed point of a given contracting element provides the main example of $[x_n]$ in the above lemma.

\subsection{Convergence boundary}
Let $X$ be compactified by  the boundary $\partial X$ with the convergence property stated as in Definition \ref{ConvBdryDefn}.
The main examples to keep in mind are all the examples in Theorem \ref{Thm:Growth-DoubleCosets2}, and the horofunction boundary in the previous subsection. In what follows, we   give a general development only according to Definition \ref{ConvBdryDefn}.

Let $o\in X$ be a fixed basepoint. For a subgroup $H<G$, let $\Lambda Ho$ denote the set of accumulation points of $Ho$ in the boundary $\partial X$.  The \textit{limit set} $[\Lambda Ho]$ of a subgroup $H$ is defined as the $[\cdot]$-locus of  $\Lambda Ho$, which by definition is  the union of the $[\cdot]$-classes of $\Lambda Ho$. The set $[\Lambda Ho]$ is independent of the choice of basepoint $o\in X$, as an escaping  sequence of segments with uniformly bounded length, which are uniformly contracting, sub-converges to the same $[\cdot]$-class  by Assumption (B).

Let $f$ be a contracting element. According to the definition, $$\gamma=\mathop{\cup}\limits_{n\in\mathbb{Z}}f^n([o,fo])$$ is a contracting quasi-geodesic. By Assumption (A), we denote  $[f^+]:=[\gamma^+]$ the $[\cdot]$-class of the limit point for the positive half-ray of $\gamma$ (i.e. where the indices are over $n>0$). Similarly, $[f^-]:=[\gamma^-]$.  It is independent of the choice of basepoint by Lemma \ref{Lem:BigThree}.
By abuse of language, we shall call the $[\cdot]$-classes $[f^-]$ and $ [f^+]$    the repelling and attracting fixed points of $f$, even though they are  possibly non-singleton.

As in Assumption (C), a $[\cdot]$-class $[\xi]$ for $\xi\in \partial X$  is called \textit{non-pinched} if for any two sequences $x_n\to [\xi]$ and $y_n\to [\xi]$, the sequence of geodesics $[x_n, y_n]$ is escaping, i.e.: misses any fixed compact set for all $n\gg 0$. A contracting element $f$ is called \textit{non-pinched} if its fixed points $[f^-], [f^+]$ are both non-pinched. Consequently, $[f^-]\cap [f^+]=\emptyset$.  We   denote by $[f^\pm]:=[f^-]\cup [f^+]$. It is proved in \cite[Lemma 3.12]{Yang22} that two non-pinched contracting elements have either the same or  disjoint fixed points. We remark that every contracting isometry on the horofunction boundary are non-pinched (\cite[Lemma 5.4]{Yang22}). However, this may fail for the action of relatively hyperbolic groups on Bowditch boundary. For instance, the natural action of $PSL(2,\mathbb Z)$ on the boundary $S^1$ of the hyperbolic plane have parabolic elements which are pinched for the action on its hyperbolic Cayley graph.  

The following result is essentially proved  in \cite[Lemma 3.19]{Yang22}. For the convenience of the reader, we briefly describe the main argument, which is useful to understand the next results. 
\begin{Lemma}\label{Lem:NSdynamics}
Let $f$ be a non-pinched contracting element with  fixed points $[f^-], [f^+]$ and with a $C$-contracting quasi-axis $\gamma$ for $C\ge 0$. Then there exist a constant $D>0$ and a  set-valued map $\pi_\gamma: \partial X\setminus [f^\pm]\to \gamma$ with the following properties: for any $\xi\ne \eta\in \partial X\setminus [f^\pm]$:
\begin{enumerate}
    \item 
    $\pi_\gamma(\xi)$ has diameter at most $D$.
    \item 
    $\pi_\gamma(f^n\xi)$ is contained in a $D$-neighborhood of $f^n\pi_\gamma(\xi)$ for any $n\in \mathbb Z$. 
    \item 
    Assume that $d(\pi_\gamma(\xi), \pi_\gamma(\eta))\ge 4D$. Let $x_n\in X\to\xi$ and $y_n\in X\to \eta$. Then $[x_n,y_n]$ intersects the  corresponding $D$-neighborhood of $\pi_\gamma(\xi)$ and $\pi_\gamma(\eta)$.
    
\end{enumerate} 
\end{Lemma}

\begin{proof}
Observe that for  any $\xi\in \partial X\setminus [f^\pm]$ and $x_n\to \xi$, the set $\{\pi_\gamma(x_n): n\ge 1\}$ is bounded: indeed, if not, Assumption (A) implies that (a subsequence of) $x_n$ tends to $[f^{\pm}]$, contradicting $\xi\notin [f^\pm]$. Take a compact neighborhood $K$ of $\xi$ with $K\subseteq X\cup \partial X\setminus [f^\pm]$. Similarly, we can prove $\pi(K):=\{\pi_\gamma(x):  x\in K\cap X\}$ is bounded. It is clear that $\pi(K_1)\subset \pi(K_2)$ for any $K_1\subseteq K_2$. We define $\pi_\gamma(\xi)$ to be the intersection of $\pi(K)$ where $K$ is taken over the compact neighborhood of $\xi$ in $X\cup\partial X$. Equivalently, it is the countable intersection of a compact neighborhood basis $K_n$ of $\xi$. It is clear that  $\pi_\gamma(\xi)$ is    non-empty, as we can take a sequence  $x_n\in K_n\to \xi$. It remains to show that $\pi_\gamma(\xi)$ has uniform diameter independent of $\xi$.

Let $L$ be a fundamental domain for the action of $\langle f\rangle$ on $\gamma$. Denote by $K$ the set of points $x\in X$ so that $\pi_\gamma(x)\cap L\ne\emptyset$. Let $\partial K=\overline{K}\setminus X$, where $\overline{K}$ is the topological closure of $K$. By a similar argument as above, $\partial K$ is  disjoint from $[f^\pm]$, so is a  compact subset in $\partial X\setminus [f^\pm]$.

By construction, as $\langle f\rangle$ acts co-compactly on $\gamma$, $\{f^n\partial K: n\in\mathbb{Z}\}$ is a uniformly locally finite cover of $\partial X\setminus [f^\pm]$. Thus,   any $\xi\in \partial X$ is contained in the union of a uniformly bounded number of members, which is compact. This implies $\pi_\gamma(\xi)$ has uniform diameter independent of $\xi$. 
\end{proof}

In the classical theory of limit sets on hyperbolic spaces,  a discrete group acts properly discontinuous on the so-called \textit{discontinuity domain}, which is complementary to the limit set.   In general convergence boundary, the next result recovers some similar feature on the fixed points of a contracting element outside the limit set. 

\begin{Lemma}\label{FixedptsConv}
Let    $H$ be any subgroup of $G$, and consider a limit point  $h_no\to \eta\in \Lambda Ho$ for $h_n\in H$. Fix a non-pinched contracting element $g\in G$ and choose any $\xi\in \Lambda E(g)o$. Then there exists a sequence of $h_n'\in H$ so that  $h_n'\xi \to [\eta]$. Moreover, if $\xi \notin [\Lambda Ho]$, then we can assume $h_n'=h_n$. 
\end{Lemma}
\begin{proof}
Consider the sequence $X_n=h_n\ax(g)$ of contracting subsets.  Up to passage to subsequence, we have the following two cases.  

\textbf{Case 1}. $X_n$ is escaping: $d(o,X_n)\to \infty$.   By assumption of $\xi\in \Lambda E(g)o$, let $k_m\in H$ so that $k_mo\to \xi.$ Thus, for each fixed $n\ge 1$, we have $h_nk_mo\to h_n\xi$ as $m\to \infty$. Observe that  if  $y_m$ is any projection point  of $h_nk_mo$ to $X_n$, then  $d(o,y_m)\to\infty$ as $m\to \infty$. Indeed, as $h_n\xi\in [\Lambda X_n]$, choose a sequence of points $z_m\in X_n$ so that $z_m\to [h_n\xi]$. Thus, the $C$-contracting property of $X_n$ implies that $d(y_m,[z_m, h_nk_mo])\le 2C$. If $\sup_{m\ge 1} d(o,y_m)<\infty$,   we then produced a non-escaping sequence of $[z_m,h_nk_mo]$ whose endpoints tending to $[h_n\xi]$. This is contradiction  as $h_n\xi$ is non-pinched. 

Now, choose   a metric $\rho$ on $\partial X$ for convenience. As $k_mo\to \xi$, we can pass to a subsequence of $k_m$ so that  $\rho(h_nk_no,  h_n\xi)\le 1/n$ on one hand; on the other hand, as $h_no\to \eta$, we can ensure   $h_nk_no\to \eta'\in [\eta]$  by Assumption (B). Taking the limit shows $h_n\xi\to \eta'\in [\eta]$.

\textbf{Case 2}. $X_n=X_m$ for any $n\ne m$. Consequently, $h_n^{-1}h_m\in E(g)$. As $h_no\to\eta$, $h_no$ forms a unbounded sequence, so  $\{h_1^{-1}h_n\}$  must be an infinite subset of $E(g)$. In particular, this  implies that  $h_1^{-1}h_n o$ accumulates into either $[g^+]$ or $[g^-]$, and by taking inverses of $h_1^{-1}h_n\in H$, we obtain $[g^\pm] \subseteq [\Lambda Ho]$.  If $\xi\in [g^\pm]$ is not contained in $\Lambda Ho$, then this case is impossible, so the proof is completed by setting $h_n'=h_n$ as desired.

As  $h_no\to \eta$ and then we have $h_1^{-1}h_no\to h_1^{-1}\eta$.  If $h_1^{-1}h_no\to [\xi]\subseteq [g^\pm]$, then $[\eta]=[h_1\xi]$. Otherwise, we need a different argument as follows. As $h_no\in X_n=X_1$, we can then find an escaping sequence $Y_n=h_n'\ax(g)$ such that $Y_n$ projects to a fixed neighborhood of $h_no\in X_n$. Assumption (A) thus implies $h_no\in X_1$ and $h_n'o\in Y_n$ both tend to $[\eta]$. As $Y_n$ is escaping, we are reduced to the Case (1). Therefore, we have $h_n'\xi\to [\eta]$ as proved there.      
\end{proof}
We  need the following variant of  \cite[Lemma 3.9]{Yang22}, where $H=G$ and $\xi$ is assumed to be in $\Lambda E(g)o$. So the result stated here is both stronger without these two assumptions, particularly allowing $H$ to consist of non-contracting elements, and weaker in assuming $g$ to be non-pinched. 
\begin{Lemma}\label{FixedptsDense}
Fix a non-pinched contracting element $g\in G$ and a basepoint $o\in X$. If    $H$ is any subgroup of $G$, then  $[\Lambda Ho]=[\overline{H\xi}]$  for any $\xi\in [g^\pm]\cap \Lambda Ho$.
\end{Lemma}
\begin{proof}
The direction $\overline{H\xi}\subseteq \Lambda Ho$ is obvious, as $\Lambda Ho$ is an $H$-invariant closed subset. For the other direction,  Lemma \ref{FixedptsConv} implies that  for any $\eta\in \Lambda Ho$, we find a sequence of $h_n'\in H$ so that $h_n'\xi\to [\eta]$. Hence, $\Lambda Ho\subset [\overline{H\xi}]$ follows. 
\end{proof}
 
Say a class $[\xi]$ is \textit{minimal} if $[\xi]=\{\xi\}$ is a singleton. We recall that all the examples in Theorem \ref{Thm:Growth-DoubleCosets2} satisfies the following assumption: the fixed points of non-pinched contracting elements of $G$ is minimal in the corresponding convergence boundary.

\begin{Corollary}\label{Cor:attractingoutside}
Suppose either the convergence boundary is the horofunction boundary or the fixed points of every non-pinched contracting element in $G$ are minimal in the convergence boundary. Assume that  $H<G$ is a  subgroup  with $[\Lambda Ho]\subsetneq [\Lambda Go]$. Then there exists a  non-pinched contracting element $f$ in $G$ so that  $[f^+]$ lies outside $\Lambda Ho$. 
\end{Corollary}
\begin{proof} 
Suppose to the contrary that the fixed points of every   non-pinched contracting elements in $G$ are contained  in $[\Lambda Ho]$. Let us fix such  element $g_0\in G$ and choose $\xi\in [g_0^+]\cap \Lambda Ho\ne\emptyset$.  Note that for each $g\in G$,  $g[\xi]=g[g_0^+]$ is one   fixed point of the non-pinched contracting element $gg_0g^{-1}$. Thus, ${G\xi}\subseteq [\Lambda Ho]$ by the above assumption.

If the fixed points of all non-pinched contracting elements are minimal, then $g[g_0^\pm]=\{g\xi\}$, so $g\xi \in \Lambda Ho$ whence $\overline{G\xi}\subseteq \Lambda Ho$.  

If $\partial X$ is the horofunction boundary, we claim  that  $\overline{G\xi}$ lies in  $[\Lambda Ho]$.   Indeed, let $g_n\xi\to \eta$ and $[g_n\xi]=[x_n]$ for $g_n\in G, x_n\in \Lambda Ho$. As $[g_n\xi]$ has  $K$-difference for $K>0$ independent of $n$, Lemma \ref{CloseRelationOverOrbit} implies that any accumulation point in $\Lambda Ho$ of $x_n$ lies in $[\eta]$. That is to say, $\eta\in [\Lambda Ho]$. 

In conclusion,  we have $\overline{G\xi}\subseteq [\Lambda Ho]$ under either of the two assumptions. 
However, by Lemma \ref{FixedptsDense},  $[\Lambda Go]\subseteq [\overline{G\xi}] \subseteq [\Lambda Ho]$. This is a contradiction. Hence, up to taking inverse, there exists $f\in G$ with $[f^+]$ outside $\Lambda Ho$.
\end{proof}
 
The above notion of the limit set can be further characterized if $H$ contains     a non-pinched contracting element. Inspired by the work \cite{MP89}, let $\Lambda H$ be the closure of the $[\cdot]$-classes of fixed points of all non-pinched contracting elements in $H$.   By \cite[Lemma 3.9]{Yang22}, these two notions of limit sets coincide up to taking $[\cdot]$-closure:  $[\Lambda Ho]=[\Lambda H]$. Moreover, it has the following desired property.

\begin{Lemma}\label{Lem:MinLimitSet}
Assume that a non-elementary subgroup $H$ contains a non-pinched contracting element. If $\Lambda$ is an $H$-invariant closed subset in $\partial X$, then $\Lambda H\subseteq [\Lambda]$. In particular,  if the partition $[\cdot]$ is trivial, then the limit set $\Lambda H=[\Lambda H]$ is the minimal $H$-invariant closed subset in $\partial X$.
\end{Lemma}
\begin{proof}
We only need to show that $\Lambda H$ is minimal: if $\Lambda$ is any $H$-invariant closed subset, then $[\Lambda H]\subseteq [\Lambda]$. Let $f$ be a non-pinched contracting element in $H$. By assumption,  $H$ is non-elementary, so the conjugates of $f$ in $H$ are all non-pinched contracting elements  with disjoint fixed points. Hence, $\Lambda$ contains at least three $[\cdot]$-classes of points, otherwise this contradicts the   North-South dynamics of two  non-pinched $f_1, f_2$ with disjoint fixed points. Choosing a point $p\in \Lambda$ different from $[f^\pm]$,   we obtain $f^np\rightarrow [f^+]$, hence $f^+\in [\Lambda]$ as $\Lambda$ is a closed subset. This holds for every non-pinched $f$,  thus $[\Lambda H]\subseteq [\Lambda]$ follows.
\end{proof}

The following result refines the statement of \cite[Lemma 3.10]{Yang22} with a similar proof.
\begin{Lemma}\label{OutsideLimitSet}
Let $h, k\in G$ be two non-pinched contracting elements so that $[h^+]\cap [k^-]=\emptyset$. Let $K$ be a closed subset in $\partial X$ so that $K\cap ([h^+]\cup [k^-])=\emptyset$. Then   for any $n\gg0$, $f_n:=h^nk^n$ are contracting elements in $G$ so that  $[f_n^\pm]\cap K=\emptyset$, and $f_n^-\to [k^-]$, and $f_n^+\to [h^+]$.
\end{Lemma} 
\begin{proof}
By \cite[Lemma 3.12]{Yang22}, $h$ and $k$ have disjoint fixed points, so the axes   $\ax(h)=E(h)o, \ax(k)=E(k)o$ have $\tau$-bounded projection for some $\tau>0$. Set $L:=d(o,h^no)$. Following the proof of \cite[Lemma 3.10]{Yang22},  consider the $(L, \tau)$-admissible path  $$\gamma:=\cup_{i\in \mathbb Z}(h^nk^n)^i([o,h^no]\cdot h^n[o,k^no])$$
where  the associated contracting sets are given by $\{(h^nk^n)^ih^n\ax(k): i\in \mathbb Z\}$. For $n\gg 0$, $\gamma$ is a contracting quasi-geodesic, so $f_n:=h^nk^n$ is contracting. By Assumption (A),  the contracting quasi-geodesic rays $\gamma^+=\cup_{n\in \mathbb N}(h^nk^n)^i([o,h^no]\cdot h^n[o,k^no])$ and $\gamma^-=\cup_{n\in \mathbb N} (h^nk^n)^{-i}([k^{-n}o,o]\cdot k^{-n}[o,h^{-n}o])$ tend to $[h^+]$ and $[k^-]$ respectively, which are closed $[\cdot]$-classes. Wew refer the reader to \cite[Lemma 3.12]{Yang22} for full details. 

It remains to show that $[f_n^+]\cap K=\emptyset$. To this end, choose an open neighborhood $U$ of $[h^+]$ so that $U\cap K=\emptyset$. We are going to prove   $[f_n^+]\subset U$ for $n\gg 0$.  The proof for the  case $[f_n^-]\cap K=\emptyset$ is similar by considering $\gamma^-$ and its projection to $\ax(k)$.

Observe that $[f_n^\pm]\cap [h^\pm]=\emptyset$. Indeed, assume to the contrary that $\xi\in [f_n^+]\cap [h^-]$ for definiteness. Note that $\gamma$ is a contracting quasi-geodesic with a common intersection   $[o,h^no]$ with $\ax(h)$. So for any point $z\in \gamma^+$, we see that the shortest point projection of $z$ to $\ax(h)$ is uniformly close to $h^no$. Similarly, any $z\in \gamma^-$ projects into a uniform neighborhood of $o$.  If $x_m\in \gamma \to [\xi]$ and $y_m\in \ax(h)\to [\xi]$, the contracting property implies that $[x_m,y_m]$ intersects a uniform neighbourhood of $[o,h^no]$ and thus a   ball centered at $o$ with radius depending on $d(o,h^no)$. This is a contradiction as $[\xi]=[h^-]$ is non-pinched by assumption. Hence, the claim is proved.    

As $n\to \infty$, $f_n^+$ tends to $[h^+]$. By Lemma \ref{Lem:NSdynamics}, for sufficiently large $n$, $[f_n^+]$ is contained in $U$. This shows   
$[f_n^+]\cap K=\emptyset$, completing the proof of the lemma.
\end{proof}

We now arrive to the main conclusion of the above discussion, which will be used in the proof of Theorem \ref{Thm:Growth-DoubleCosets-Boundary-SecType}.
\begin{Corollary}\label{FixedpointsOutsideLimitset}
Under the assumption of Corollary \ref{Cor:attractingoutside}, assume that an infinite  subgroup $H<G$ has a proper limit set $[\Lambda Ho]\subsetneq [\Lambda Go]$. Then  there exist infinitely many contracting elements in $G$ whose fixed points are pairwise distinct and outside $[\Lambda Ho]$.
\end{Corollary}
\begin{proof} 
Let $f\in G$ be a non-pinched contracting element provided by Corollary \ref{Cor:attractingoutside} so that $[f^+]\cap \Lambda Ho=\emptyset$ and hence $[f^+]\cap [\Lambda Ho]=\emptyset$. We claim that we can choose  two non-pinched contracting elements $h, k\in G$ with fixed points $h^-,k^+\notin [\Lambda Ho]$ and $[h^+]\cap [k^-]=\emptyset$. 

Indeed, if $[f^-]$ is outside $\Lambda Ho$ as well, we just choose $h=g$ and $k=g^{-1}$. We now assume that $[f^-]$ is contained in $[\Lambda Ho]$ but $[f^+]\cap \Lambda Ho=\emptyset$. 
As $H$ preserves $[\Lambda Ho]$, we have that $h[f^-]\subseteq [\Lambda Ho]$ and $h[f^+]$ is outside $\Lambda Ho$ for any $h\in H$. Recall that two non-pinched contracting elements have either disjoint fixed points or the same fixed points. Note that $h[f^+]$ are fixed points of $hfh^{-1}$. If $h[f^-]=[f^-]$ for any $h\in H$, then we must obtain that $h[f^+]=[f^+]$. In particular,  $H$ is a subgroup of the set-stabilizer of $[f^\pm]$. However, the stabilizer of the fixed points of a non-pinched contracting element $f$ coincides with the maximal elementary subgroup $E(f)$. If $H<E(f)$ is infinite, the $[\cdot]$-closure of the limit set $\Lambda Ho$ is the one $[f^\pm]$ of $E(f)$. This contradicts the assumption $[f^+]\cap [\Lambda Ho]=\emptyset$. Hence, $h[f^-]\ne [f^-]$ for some $h\in H$ and then $h[f^+]\ne[f^+]$. Setting $h=f$ and $k=hfh^{-1}$ completes the verification of the claim.

By Lemma \ref{OutsideLimitSet} for $K:=\Lambda Ho$, we find a sequence of  contracting elements $f_n:=h^nk^n$  with $f_n^-\to [h^-]$ and $f_n^+\to [k^+]$, and $[f_n^\pm]$ lies outside $\Lambda Ho$ for all $n\gg 0$.  This is what we wanted. 
\end{proof}

At last, we compare the Morse subgroups and subgroups of second kind on the horofunction boundary.
By \cite[Theorem 1.1]{Yang22}, the horofunction boundary is a convergence boundary for the finite difference partition, where all boundary points are non-pinched. This implies that every contracting element is non-pinched on the horofunction boundary. 
\begin{Corollary}\label{Cor:Morse2ndtype}
Suppose that $\partial X$ is the horofunction boundary of $X$. Let $H$ be  a Morse subgroup of infinite index. Then $[\Lambda Ho]$ is a proper subset of $[\Lambda Go]$.
\end{Corollary}
\begin{proof}
By Lemma \ref{Lem:BddProj}, there exists a non-pinched contracting element $g\in G$ such that the projection  of $Ho$ to $\ax(g)$ is a bounded set, denoted by $B$. We claim that $[\Lambda H]$ is disjoint with $[g^\pm]$. If not, let $p\in \Lambda Ho\cap [g^\pm]$, so  $h_no\to p$ for some sequence $h_n\in H$ and $g_no\to [p]$ for some sequence $g_n\in \langle g\rangle$. As $[p]$ is non-pinched, $[h_no, g_no]$ is escaping: $d(o,[h_no,g_no])\to\infty$. On the other hand, the $C$-contracting property of $\ax(g)$ shows that $[h_no, g_no]$ intersects the $C$-neighborhood of $B$. This is a contradiction, so $[\Lambda Ho]\subsetneq [\Lambda Go]$ follows.
\end{proof}

\section{Double coset growth for Morse subgroups}\label{SecMorseGrowth}

This section is devoted to the proof of Theorem \ref{Thm:Growth-DoubleCosets}.


Suppose that $H$ and $K$ are  Morse subgroups in $G$ of infinite index. Let $g_H, g_K\in G$ be  contracting elements  and $\tau>0$  given by Corollary    \ref{Lem:BddProj} such that
\begin{equation}\label{BddProjEQ}
\begin{array}{cc}
\forall h\in H:\; & \mathrm{diam}(\pi_{\mathrm{Ax}(g_H)}([o,ho]))\le \tau\\
\forall k\in K:\; & \mathrm{diam}(\pi_{\mathrm{Ax}(g_K)}([o,ko]))\le \tau 
\end{array}
\end{equation} 
We can actually choose $g_H=g_K$.

As $G$ contains   infinitely many pairwise independent contracting elements, let us fix a set $F$ of three pairwise independent contracting elements which are all independent with $g$. Form the following $C$-contracting system   $$\mathbb A=\{h\ax(f): f\in F\cup\{g_H, g_K\}, h\in G\}$$ for some $C>0$, which consists of all translated axis of $F\cup \{g_H, g_K\}$ under $G$. Let $L,\tau, \Lambda=\Lambda(C,\tau)$ be constants satisfying Lemma \ref{Lem:Extension} for this $F$. 

For a sequence of elements $\{g_i\}_{1\leq i\leq n}$ in $G$, we define the path \textit{labelled by} $g_1g_2\cdots g_n$ as the concatenation of geodesic segments
$$[o,g_1o]\cdot g_1[o,g_2o]\cdot\cdots (g_1\cdots g_{n-1})[o,g_no].$$

\begin{Lemma}\label{ExtensionClaim}
There exists $M>0$ such that the following holds.
For any $t\in G$, there exist $a, b\in F$ such that the path
labelled by the word $s:=g_H^M\cdot 1\cdot a\cdot t\cdot b\cdot 1\cdot g_K^M$ is $(L,\tau)$-admissible, where the contracting subsets are provided by the appropriately translated axes of $g_H, g_K$ and $a, b\in F$.
\end{Lemma}
\begin{proof}
By Lemma \ref{Lem:Extension}, for the pairs $(g_H^M, t)$ and $(t, g_K^M)$, there exist $a, b\in F$ such that $g_H^M\cdot a\cdot t$  and $t\cdot b \cdot  g_K^M$ both label $(L, \tau)$-admissible paths where $L=\max\{d(o, g_H^Mo),d(o, g_K^Mo)\}$. By Definition \ref{AdmDef} of admissible paths, the concatenation of these two paths is  still an $(L, \tau)$-admissible path. The proof is complete.
\end{proof}

Set $r_0:=2M |g|+2\max\{|f|: f\in F\}$ and $m:=\sharp B_G(r-r_0)=\mathrm{gr}_G(r-r_0)$. Let us list the elements $$B_G(r-r_0)=\{t_1,t_2,\cdots,t_m\}.$$

By Lemma \ref{ExtensionClaim}, for $1\leq i\leq m$, the path labelled by the word $$s_i:=g_H^Ma_it_ib_ig_K^M$$ for some $a_i,b_i\in F$ is an $(L,\tau)$-admissible path. Thus,
$||s_i|-|t_i||\le r_0$ for each $s_i$.

Recall that $\mathbb D(H,K)=\{HgK: g\in G\}$ is the collection of double cosets about $H$ and $K$. 
The key result to estimate the lower bound of double cosets is the injectivity of the following map $\Pi:G\to \mathbb D(H,K)$. 

\begin{Lemma}\label{Lem:InjectiveOnBall}
There exists a constant $N_0>0$ such that the following map 
\begin{align*}
\Pi: &B_G(r-r_0)\to \mathbb D(H,K)\\
     &t_i\mapsto Hs_iK
\end{align*}
is at most $N_0$-to-1.
\end{Lemma}
\begin{proof}
Let $Hs_iK=Hs_jK$ so that $t_i\neq t_j\in B_G(r-r_0)$. Then there exist $h\in H,k\in K$ so that $s_j=hs_ik$. Since we have $s_i\ne s_j$, at least one of  $h, k$ is nontrivial. For definiteness, assume that $h\ne 1$ and the other case $k\ne 1$ is symmetric.  Then the following word  
\[W:=g_K^{-M}b_j^{-1}t_j^{-1}a_j^{-1}\cdot g_H^{-M}hg_H^M\cdot a_it_ib_i\cdot g_K^Mk\]
represents the trivial element $1=s_j^{-1}hs_ik.$
The goal is, with a fixed number of exceptions of $h$, to prove  that the path labelled by $W$  is an $(L,\tau)$-admissible path. 

Recall that by Lemma \ref{ExtensionClaim}, the paths labeled by $s_j^{-1}$ and $s_i$ (thus their subpaths) are already $(L,\tau)$-admissible paths. By Remark \ref{ConcatenationAdmPath}, we are left  to check that both $g_H^{-M}hg_H^M$ and $g_K^Mk$ label  $(L,\tau)$-admissible paths, so their concatenation gives the desired $(L,\tau)$-admissible path.

By Corollary \ref{Lem:BddProj}, we have  
\[\max\{\mathrm{diam}(\pi_{\ax(g_H)}[o,ho]),\mathrm{diam}(\pi_{\ax(g_H)}[o,h^{-1}o])\}\leq \tau.\]

As $E(g)$ contains $\langle g\rangle$ as a finite index subgroup, $H\cap \langle g\rangle=\{1\}$ implies  that $E_0=H\cap E(g)$ is finite. If $h\notin E_0$, then we have $hE(g)\neq E(g)$, and hence the word $g_H^{-M}hg_H^{M}$ labels an $(L,\tau)$-admissible path. Similarly, if  $k\notin E_1:=K\cap E(g)$ we know that $g_K^Mk$ also labels an $(L,\tau)$-admissible path. Hence,   the above word $W$ labels an $(L,\tau)$-admissible path $\gamma$, which is thus a $\Lambda$-quasi-geodesic by Lemma \ref{Lem:Extension} for some $\Lambda>1$. By construction,   the word $W$ contains two copies of subwords $g_H^M$ and $g_K^M$, so  $\gamma$ has    length  bounded below by $2|g_H^M|+2|g_K^M|$. On the other hand, as $W$ represents the trivial element, the  $\Lambda$-quasi-geodesic $\gamma$ have the same endpoints. As a result, 
\[2|g_H^M|+2|g_K^M| \leq  \Lambda\]
contradicting the choice of $M$. Therefore, if $\Pi$ fails to be injective, then we must have $h\in E_0=H\cap E(g)$ or  $k\in E_1=K\cap E(g)$. Hence, setting $N_0=\sharp E_0\sharp E_1$, we proved that the map $\Pi$ is at most $N_0$-to-one. 
\end{proof}
 
Thus, the first statement of the theorem follows. If $G$ has purely exponential growth, then  $\sharp B_G(r) \asymp \omega_G^r$ for any $r>0$, the ``moreover" statement follows.

\section{Double coset growth for subgroups of second type}\label{SecSecondType}

The goal of this section is to prove Theorem \ref{Thm:Growth-DoubleCosets-Boundary-SecType}. We proceed  as the proof of Theorem \ref{Thm:Growth-DoubleCosets}, where   Corollary \ref{Lem:BddProj} is replaced by the following  lemma.

\begin{Lemma} \label{Lem:Proj-Int-Bounded}
Suppose that $H$ is a subgroup of $G$, and   $g$ is a contracting element  with $[g^-],[g^+]$ outside $\Lambda H o$ for some $o\in X$.
Then there exists a constant $D>0$ so that for any $h\in H$, we have
   $$\mathrm{diam}(\pi_{\ax(g)}([o,ho]))\leq D.$$
In particular, the intersection  $H\cap E(g)$   is finite.
 \end{Lemma}

\begin{proof}
Recall that $\ax(g)=E(g)o$ is a $C$-contracting quasi-axis of $g$ for some $C>0$. 
Suppose to the contrary that there exists a sequence of $h_i\in H$ so that \[\lim\limits_{i\rightarrow+\infty}\mathrm{diam}(\pi_{\ax(g)}([o,h_io]))=+\infty.\]
Denote $\gamma_i:=[o,h_io]$. Let $y_i$ be the exit point of $\gamma_i$ in $N_C(\ax(g))$:  $d(y_i,\ax(g))=C$ and $d(y,\ax(g))>C$ for any $y\in(y_i,h_io]_{\gamma_i}$. The $C$-contracting property thus implies 
$$\mathrm{diam}(\pi_{\ax(g)}([y_i,h_io]))\le C$$
so
\[d(y_i,\pi_{\ax(g)}(h_io))\leq d(y_i,\ax(g))+\mathrm{diam}(\pi_{\ax(g)}([y_i,h_io]))\leq2C.\]
On the one hand, as 
$d(o,y_i)=\mathrm{diam}(\gamma_i\cap N_{C}(\ax(g)))\rightarrow+\infty,$   the assumption (A) implies that $h_io$ tends to $[g^+]$. On the other hand, the accumulation points of $h_io$ is contained in $\Lambda Ho$. This contradicts  $[g^+]\cap \Lambda Ho=\emptyset$. The proof is then complete.
\end{proof}


\begin{proof}[Proof of Theorem \ref{Thm:Growth-DoubleCosets-Boundary-SecType}] 
Since $[\Lambda Ho],[\Lambda Ko]$ are proper subsets of $[\Lambda Go]$, there exist  contracting elements $g_H, g_K$ with $[g_H^-],[g_H^+]$ outside $[\Lambda Ho]$ and $[g_K^-],[g_K^+]$ outside $ [\Lambda Ko]$ by Corollary \ref{FixedpointsOutsideLimitset}.
By  Lemma \ref{Lem:Proj-Int-Bounded} replacing Corollary \ref{Lem:BddProj}, we again have the bounded projection bound  as in (\ref{BddProjEQ}).

The rest of the proof proceeds exactly  as that of Theorem \ref{Thm:Growth-DoubleCosets}, which uses only the above bounded projection properties, without involving the Morseness of $H,K$ at all.
\end{proof}

As a product of the above proof, we present  the proof of the following.
\begin{proof}[Proof of Proposition \ref{freeproductcombProp}]
The goal is to prove that   the natural epimorphism $$H\ast gKg^{-1}\to \Gamma$$ where $\Gamma$ is generated by  $H, gKg^{-1}$ is injective.   To that end, consider an alternating word $W=h_1 gk_1g^{-1}\cdots h_n gk_ng^{-1}$, where $h_i\in H\setminus 1, k_i\in K\setminus 1$. By (\ref{BddProjEQ}) and setting $L=d(o,go)$, the word $W$ labels an $(L, \tau)$-admissible path $\gamma$,  which is thus a $\Lambda$-quasi-geodesic by Lemma \ref{Lem:Extension} for some $\Lambda>1$ independent of $L$. Up to replacing $g$ with high powers,  the constant $L$ can be arbitrarily large so  the two endpoints of $\gamma$ are distinct. Consequently, each non-empty reduced word $W$ represents a non-trivial element, completing the proof of the injectivity. 
\end{proof}

\section{Generic 3-manifolds are hyperbolic: Theorem \ref{Thm:HypExpGen}}\label{SecApp3Mfds}

Suppose that the mapping class group $\mathrm{Mod}(\Sigma_g)$ acts on the Teichm\"{u}ller space $(\mathcal{T}(\Sigma_g),d_\mathcal{T})$.
Let $H<\mathrm{Mod}(\Sigma_g)$ be the handlebody subgroup, so by \cite[Theorem 1.2]{Masur86},
$\Lambda H\neq\Lambda\mathrm{Mod}(\Sigma_g)$ as described in Subsection \ref{subsec:App-3Man}. Hence $\mathrm{gr}_{H,H}(r)\geq \delta \mathrm{gr}_G(r)$ for any $r\ge 1$ by Corollary \ref{DCG-Mod}. 






\subsection{Shadowing maps}
Let   $\mathcal{C}(\Sigma_g)$ be the curve graph of $\Sigma_g$ where  the vertex set consists of (isotopy classes of) essential simple closed curves and two vertices are adjacent if they can be represented as disjoint essential simple closed curves.

In proving the hyperbolicity of the curve graph, Masur-Minsky \cite{MasMin99} built a  Lipschitz map up to bounded error (called \textit{coarsely Lipschitz}) 
$$\phi_1:(\mathcal{T},d_{\mathcal{T}}){\longrightarrow} (\mathcal{C} (\Sigma_g), d_{\mathcal{C}})$$
which sends a marked hyperbolic metric  to one of the  systoles of $\Sigma_g$ (i.e. the shortest simple closed curves).  The crucial property of this map is that any Teichm\"{u}ller geodesic is sent to a un-parameterized quasi-geodesic in the curve graph. It is known that a pseudo-Anosov element $f$ acts by translation on a  contracting Teichm\"{u}ller geodesic $\gamma$ in $\mathcal{T}(\Sigma_g)$ by \cite{Minsky96}, and acts loxodromically on the curve graph by \cite{MasMin99}. Thus, the image of $\gamma$ under $\phi$ is a re-parametrized quasi-geodesic $\gamma_1$.

Let $\mathcal{D}(V)$ be the collection of essential simple closed curves on $\partial V=\Sigma_g$ which bound discs in the handlebody $V$. By Masur-Minsky \cite{MasMin04}, the disk set $\mathcal{D}(V)$ is a quasi-convex subset in $ \mathcal{C}(\Sigma_g)$, so $\{g\cdot \mathcal{D}(V): g\in \mathrm{Mod}(\Sigma_g)\}$ forms a collection of uniformly quasi-convex subsets. The \textit{electrified disk complex} ($\mathcal{D}(\Sigma_g)$, $d_{\mathcal{D}}$) is obtained from $\mathcal{C}(\Sigma_g)$ by adding an edge between any two vertices in each $\mathrm{Mod}(\Sigma_g)$-translate of  $\mathcal{D}(V)$.

Endowed with length metric, the identification of vertex sets  defines a coarsely Lipschitz map 
$$
\phi_2: (\mathcal{C}(\Sigma_g), d_{\mathcal{C}})
{\longrightarrow}(\mathcal{D}(\Sigma_g),d_{\mathcal{D}})
$$
By \cite[Prop 7.12]{Bowditch12}, the  electrified disk complex $\mathcal D(\Sigma_g)$ is a hyperbolic space. Furthermore,  by \cite[Prop 2.6]{KapRafi14}, a geodesic between two points $x,y$ in $\mathcal{C}(\Sigma_g)$ is contained in a uniformly finite neighborhood of a geodesic between $x,y$ in $\mathcal D(\Sigma_g)$. In other words, $\phi_2$ also sends geodesics to un-parametrized quasi-geodesics. Note that $\phi_2$ identifies the vertex set  of two complexes but $\mathcal{D}(\Sigma_g)$ contains more edges.

In summary, the composition map $\Phi:=\phi_2\circ\phi_1: (\mathcal{T},d_{\mathcal{T}})\longrightarrow (\mathcal{D}(\Sigma_g),d_{\mathcal D})$ is
a coarsely Lipschitz map, sending a Teichm\"{u}ller geodesic to a un-parameterized $\lambda$-quasi-geodesic in the electrified disk complex.  See more details  in   \cite{KapRafi14,MahSch21}.

Since $f$ acts as a loxodromic element on both complexes, it makes definite progress in both of them. Let $\mathbb X=\{g\textrm{Ax}(f): g\in \mathrm{Mod}(\Sigma_g)\}$, where $\textrm{Ax}(f)$ is the Teichm\"{u}ller axis. 
\begin{Lemma}\label{ShadowProgressLem}
Let $f$ be a pseudo-Anosov element that acts loxodromically on $\mathcal D(\Sigma_g)$. Then there exist  $L=L(f)$   and  $\kappa=\kappa(f)>0$ with the following property. Let $\gamma$ be any geodesic segment with endpoints in $Go$, and $\mathbb X(\gamma) :=\{X\in\mathbb X: \|\gamma\cap N_C(X)\|>L\}$.  Then $$d_{\mathcal D}(\gamma_-,\gamma_+)> \kappa\sum_{X\in\mathbb X(\gamma)}\ell(N_C(X)\cap \gamma)$$ 
where $\ell$ denotes the length of paths in the Teichm\"{u}ller metric.
\end{Lemma}
\begin{proof}
Let $\gamma_1$  be the image of $\textrm{Ax}(f)$ by $\phi_1$. This is a (parametrized) quasi-geodesic in $\mathcal{C}(\Sigma_g)$ on which $f$ acts by translation. To be precise, there exist $\lambda, c$ such that  for any $x, y\in \textrm{Ax}(f)$, we have
$$
\lambda^{-1} d_{\mathcal C}(\phi_1(x),\phi_1(y))
-c\le d_{\mathcal T} (x,y) \le \lambda d_{\mathcal C}(\phi_1(x),\phi_1(y))
+c$$
Note that the   constants  $\lambda, c$ are uniform independent of $f$. 

Similarly, let $\gamma_2$ be the image of   $\gamma_1$ under $\phi_2$ which is  a (parametrized) quasi-geodesic.   

Since $\mathbb X$ has bounded intersection, the images under $\Phi=\phi_1\circ\phi_2$ of any two distinct Teichm\"{u}ller axes in $\mathbb X$ have bounded intersection as well, while the image of each Teichm\"{u}ller axis is a parametrized quasi-geodesic. If $L$ is chosen big enough, the image of $N_C(X)\cap \gamma$ is long enough compared with the bounded overlap. Hence, the desired lower bound follows in a straightforward way.   
\end{proof}

By  \cite[Theorem 8.3]{MahSch21}, if a pseudo-Anosov element $f\in\mathrm{Mod}(\Sigma_g)$ has   neither its stable nor unstable laminations  in the closure of  $\mathcal {D}(V)$, then $f$   acts loxodromically on $\mathcal{D}(\Sigma_g)$. Let us now fix such a pseudo-Anosov element $f$.

\subsection{Linear growth of Heegaard/Hempel distance: Proposition \ref{Prop:Hemple-Nondeg}}
This subsection is devoted to the proof of Proposition \ref{Prop:Hemple-Nondeg}, which is the key ingredient of Theorem \ref{Thm:HypExpGen}.

To that end, by Lemma \ref{Lem:ExpGen}, we only need to check those double cosets $H\varphi H$ with a representative $\varphi\in\mathcal{W}_{\epsilon,M,f}^{\theta,L}$ for some fixed constants $\epsilon,M,0<\theta\leq\frac{1}{2},L=L(\theta)$, where $f$ is chosen below.

\begin{Lemma}
The set of elements $\varphi\in \mathrm{Mod}(\Sigma_g)$ so that $d_{\mathcal{D}}(\Phi(o),\varphi\Phi(o))>\kappa d(o,\varphi o)$ is exponentially generic.
\end{Lemma}
\begin{proof}
As the set of barrier-free elements is exponentially negligible by Proposition \ref{pro:growthtight}.(2), we can assume that $[o,\varphi o]$ contain an $(r,f)$-barrier. Let $\mathbb B$ be the set of all maximal subsegments that are contained in $N_r(X_i)$ for some $X_i\in \mathbb X$. By taking a large enough power, we can choose a pseudo-Anosov element $f$ satisfying Lemma \ref{ShadowProgressLem} and
$$d_{\mathcal T}(o,fo)>2L+4D+8r$$ where $D$ is the bounded intersection constant of $\mathbb X$. As  each $\beta\in \mathbb B$ is an $(\epsilon, f)$-barrier, we have $\ell(\beta)>d_{\mathcal T}(o,fo)-2r$. By the $D$-bounded intersection, any two distinct segments $\beta_1,\beta_2$ in $\mathbb B$ have an overlap at most $D\le 1/4 \min\{\ell(\beta_1),\ell(\beta_2)\}$. Thus 
$$
\sum_{\beta\in \mathbb B} \ell(\beta) \le  4\cdot \ell(\cup_{\beta\in \mathbb B}\beta)
$$
where $\ell(\cdot)$ denotes the length in Teichm\"{u}ller metric.

Let $\mathbb K_0$ be the set of components in the complement $[o,\varphi o]\setminus \cup\mathbb B$ of the union of $\mathbb B$. By construction, each segment $\alpha$ in $\mathbb K_0$ is $(r,f)$-barrier-free. Let $\mathbb K=\{\alpha\in\mathbb K_0: \ell(\alpha)>L\}$. If 
$\sum_{\alpha\in \mathbb K} \ell(\alpha) \ge \theta d_{\mathcal T}(o,\varphi o)$    
then such elements $\phi$ are contained in $\mathcal{V}_{\epsilon,M,f}^{\theta,L}$, which is a exponentially negligible subset by Proposition \ref{pro:growthtight}.(3).  Up to ignoring this subset, we can assume that $\sum_{\alpha\in \mathbb K} \ell(\alpha) < \theta d_{\mathcal T}(o,\varphi o)$, so $$
\sum_{\beta\in \mathbb B} \ell(\beta) +\sum_{\alpha\in \mathbb K_0 \setminus \mathbb K} \ell(\alpha) \ge (1-\theta) d_{\mathcal T}(o,\varphi o)
$$

We need to analyze the contribution of $\mathbb K_0\setminus \mathbb K$ which consists of   $(r,f)$-barrier-free segments $\alpha$ of length at most $L$. Each such $\alpha$ with length $\le L$ must be adjacent to an $(\epsilon, f)$-barrier $\beta \in \mathbb B$ with length  $\ge 2L$. Thus, $$
\sum_{\alpha\in \mathbb K_0 \setminus \mathbb K} \ell(\alpha) \le \frac{1}{2}\sum_{\beta\in \mathbb B} \ell(\beta)
$$
From the above estimates, we obtain that $\ell(\cup_{\beta\in \mathbb B}\beta) \ge (1-\theta)d_{\mathcal T}(o,\varphi o)/6$. The conclusion then follows from Lemma \ref{ShadowProgressLem}.
\end{proof}

By definition,  the Heegaard/Hempel distance $d_{H}(V_1\cup_{\varphi} V_2)$ of a Heegaard splitting $V_1\cup_{\varphi} V_2 $ is bounded below by  $d_{\mathcal{D}}(\Phi(o), \varphi(\Phi(o)))$ in $\mathcal{C}_{\mathcal{D}}(\Sigma_g)$.
Consequently, the set \[
\mathcal {H}=\{H\varphi H: d_{H}(V_1\cup_{\varphi} V_2)\geq\kappa d_\mathcal{T}(o,H\varphi Ho)\}\]
is exponentially generic in $\{H\varphi H:\varphi\in\mathrm{Mod}(\Sigma_g)\}$.

\subsection{Completion of the  proof of Theorem \ref{Thm:HypExpGen} }
By work of Hempel \cite[Corollaries 3.7, 3.8]{Hempel01} and Perelman's work on Thurston's geometrization conjecture, the 3-manifold $M_\varphi$ is hyperbolic if the Heegaard/Hempel distance $d_{H}(V_1\cup_{\varphi} V_2)\geq 3$. The main result of \cite[Corollary]{SchTom06} says that, if $d_{H}(V_1\cup_{\varphi} V_2)>2g$, any two Heegaard surfaces of  genus $g$ must be isotopic, and the Heegaard genus of $M_\varphi$ is equal to $g$.
Thus, if $M_\varphi$ and $M_{\varphi'}$ are homeomorphic, then $H\varphi H=H\varphi'H$ and the Heegaard genus of $M_\varphi$ is $g$. 
As a result, the map
\begin{align*}
\{H\varphi H:\varphi\in\mathrm{Mod}(\Sigma_g)\} &\longrightarrow \{M_\varphi :\varphi\in\mathrm{Mod}(\Sigma_g)\}\\
H\varphi H&\longmapsto M_\varphi
\end{align*} 
is injective on the exponentially generic collection $\mathcal {H}$ of double cosets, so that the image is contained in the following
\[\Gamma_{g}=\{M_\varphi :M_\varphi~\text{is hyperbolic with Heegaard genus}~ g\}\]

Hence, we proved that $\Gamma_g$ is an exponentially generic subset of $\Delta_{g}$, with respect to {geometric complexity}.

\bibliographystyle{plain}
\bibliography{Ref}


\end{document}